\documentclass[a4paper,11pt]{amsart}
\usepackage{amsmath}%
\usepackage{amsfonts}%
\usepackage{amssymb}%
\usepackage{graphicx}%

\usepackage{comment}%    for commenting out when revising
\usepackage{tikz}%			to make pictures  
\usepackage{caption}%    for captioning without number
\usepackage{mathtools} % for \coloneqq
\usepackage{scalerel,stackengine} %for reallywidehat
\usepackage{bm}% for bold symbols, especially greek letters

\theoremstyle{plain}
\newtheorem{theorem}{Theorem}
\newtheorem{corollary}[theorem]{Corollary}
\newtheorem{lemma}[theorem]{Lemma}
\newtheorem{notation}[theorem]{Notation}
\newtheorem{proposition}[theorem]{Proposition}

\theoremstyle{definition}
\newtheorem{definition}[theorem]{Definition}

\theoremstyle{remark}

\theoremstyle{remark}
\newtheorem{remark}[theorem]{Remark}
 
\theoremstyle{remark}

\newcommand{\mathsc}[1]{\textnormal{\textsc{#1}}} % so textsc behaves in formulas when in bold; reverts to nonbold
\newcommand{\Xb}{X_{\mathsc{Ban}}}
\newcommand{\UFs}{U({F_{\mathsc{sing}}})}
\newcommand{\UFhat}{U({\widehat{F}_{\mathsc{sing}}})}
\newcommand{\Fsing}{{F_{\mathsc{sing}}}}
\newcommand{\Fhat}{{\widehat{F}_{\mathsc{sing}}}}

\newcommand{\Fnormwidehat}{\reallywidehat{\mathsc{Nrm}(\Fsing)}}
\newcommand{\GVban}{n^0_{\beta_{\mathbf{d}}}(\Xb)}
\newcommand{\naive}{\widetilde{n}^0_{\beta_{d_1,d_2}}(\Xb)}
\newcommand{\naivev}{\widetilde{n}^0_{\beta_{\mathbf{d}}}(\Xb)}
\newcommand{\Cban}{C_{\mathsc{Ban}}}

\newcommand{\Msing}{{M_{\mathsc{sg}}}}
\newcommand{\MsingT}{{M_{\mathsc{sg}}^T}}
\newcommand{\MsingTP}{{M_{\mathsc{sg}}^{\text{TP}}}}
\newcommand{\Mhat}{{\widehat{M}_{\mathsc{sg}}}}
\newcommand{\MhatTP}{{\widehat{M}_{\mathsc{sg}}^{\text{TP}}}}
\newcommand{\MU}{{M}_{\mathsc{usg}}^T}
\newcommand{\Waff}{{W}_{\mathsc{Aff}}}

\newcommand{\EEE}{\widetilde{\EE}}
\newcommand{\FFF}{\widetilde{\FF}}
\newcommand{\pr}{\textit{pr}}
\newcommand{\HHom}{\operatorname{\mathcal{H}\! \mathit{om}}}

\newcommand{\CC}{\mathbf{C} }

\newcommand{\PP}{\mathbf{P} }
\newcommand{\RR}{\mathbf{R} }

\newcommand{\ZZ}{\mathbf{Z}\,}
\newcommand{\CCC}{\mathcal{C} }
\newcommand{\DDD}{\mathcal{D} }
\newcommand{\EE}{\mathcal{E} }
\newcommand{\FF}{\mathcal{F} }
\newcommand{\GG}{\mathcal{G} }
\newcommand{\II}{\mathcal{I}\,}
\newcommand{\LL}{\mathcal{L}\,}

\newcommand{\NNN}{\mathcal{N}}
\newcommand{\OO}{\mathcal{O} }

\newcommand{\Coh}{\operatorname{Coh}}

\newcommand{\Ext}{\operatorname{Ext}}

\newcommand{\HH}{\operatorname{H}}
\DeclareMathOperator{\Ker}{Ker}
\newcommand{\Pic}{\operatorname{Pic}}
\newcommand{\Hom}{\operatorname{Hom}}

\newcommand{\Def}{\operatorname{Def}}

\newcommand{\Bl}{\operatorname{Bl}}

\newcommand{\Sym}{\operatorname{Sym}}
\newcommand{\Supp}{\operatorname{Supp}}

\newcommand{\OPD}{\operatorname{OPD}}
\newcommand{\ODOP}{\operatorname{ODOP}}
\newcommand{\Spec}{\operatorname{Spec}}
\newcommand{\Tot}{\operatorname{Tot}}
\newcommand{\wec}{Behrend function weighted Euler characteristic}

\makeatletter
\newcommand*\bigcdot{\mathpalette\bigcdot@{.5}}
\newcommand*\bigcdot@[2]{\mathbin{\vcenter{\hbox{\scalebox{#2}{$\m@th#1\bullet$}}}}}
\makeatother

\newcommand\sbullet[1][.5]{\mathbin{\vcenter{\hbox{\scalebox{#1}{$\bullet$}}}}}

\usetikzlibrary{calc,decorations.pathmorphing,shapes,cd}

\newcounter{sarrow} % ensures that each node has a unique name

\stackMath
\newcommand\reallywidehat[1]{%
\savestack{\tmpbox}{\stretchto{%
  \scaleto{%
    \scalerel*[\widthof{\ensuremath{#1}}]{\kern.1pt\mathchar"0362\kern.1pt}%
    {\rule{0ex}{\textheight}}%WIDTH-LIMITED CIRCUMFLEX
  }{\textheight}% 
}{2.4ex}}%
\stackon[-6.9pt]{#1}{\tmpbox}%
}
\title[GV invariants of the Banana manifold]{Genus Zero Gopakumar-Vafa invariants of the Banana manifold}

\begin{document}

\date{\today}

\author{Nina Morishige}
\address{Nina Morishige, Department of Mathematics, The University of British Columbia, Vancouver, BC, V6T 1Z2 Canada}%
\email{nina@math.ubc.ca}%

\begin{abstract}
The Banana manifold $\Xb$ is a compact Calabi-Yau threefold constructed as the conifold resolution of the fiber product of a generic rational elliptic surface with itself, first studied in \cite{bryan19}. We compute Katz's genus 0 Gopakumar-Vafa invariants \cite{katz08} of fiber curve classes on the Banana manifold $\Xb\to \PP^1$. The weak Jacobi form of weight -2 and index 1 is the associated generating function for these genus 0 Gopakumar-Vafa invariants. The invariants are shown to be an actual count of structure sheaves of certain possibly nonreduced genus 0 curves on the universal cover of the singular fibers of $\Xb\to\PP^1$. 
\end{abstract}

\maketitle
%%%%%%%%%%%%%%%%%%%%%%%%%%%%%%%
%\input{./section0}
\section{Introduction}\label{sec: intro}

\subsection{Background}

The genus zero Gopakumar-Vafa invariants are integer valued deformation invariants of Calabi-Yau threefolds that appeared in physics as a virtual count of rational curves on $X$ \cite{gopakumar-vafa98}. 

Mathematically Katz defined the genus 0 Gopakumar-Vafa invariants as follows \cite{katz08}.
\begin{definition} \label{def:MXbeta}
Let $X$ be a projective Calabi-Yau threefold over $\CC$, together with a fixed curve class $\beta\in H_2(X)$. By a Calabi-Yau threefold $X$, we mean a smooth threefold with trivial canonical bundle $K_X\cong \OO_X$. 
We define $M^X_{\beta}$ to be the moduli space of Simpson semistable \cite{simpson94} pure 1-dimensional sheaves $\FF$ on $X$ with $\text{ch}_2(\FF) = \beta^{\vee}$ and $\chi(\FF) = 1$.
\end{definition} 
\begin{definition} 
The genus 0 Gopakumar-Vafa (GV) invariants $n^0_{\beta}(X)$ of $X$ in curve class $\beta$ are defined as the Behrend function weighted Euler characteristics of this moduli space:
\begin{equation}
n^0_{\beta}(X) = e(M^X_{\beta},\nu) \coloneqq \sum_{k\in\ZZ}{k\cdot e_{\textsl{top}}(\nu^{-1}(k))}
\label{def:GV}
\end{equation}
where $e_{\textsl{top}}$ is topological Euler characteristic and $\nu:M^X_{\beta}\rightarrow \ZZ$ is Behrend's constructible function \cite{behrend09}.
\end{definition}

\begin{remark}
The moduli space $M^X_{\beta}$ contains no strictly semi-stable sheaves, and so the moduli space is a projective scheme. (See Lemma \ref{lemma: nosemistable-euler}).
\end{remark}

\begin{remark}
The stability condition is equivalent to a condition on the Euler characteristic, namely, that a coherent sheaf $E\in M^X_{\beta}$ is stable if and only if any subsheaf $E'\subset E$ has nonpositive Euler characteristic $\chi(E')\leq 0$. This makes the moduli space manifestly independent of the choice of an ample class. (See Lemma \ref{lemma: nosemistable-euler}).
\end{remark}
More recently, an interpretation of all genus GV invariants $n^g_{\beta}(X), g\geq 0$, in terms of a sheaf of vanishing cycles on $M^X_{\beta}$ is given in \cite{maulik-toda18}. In the case of genus 0 invariants, this reduces to the previous definition. Toda  \cite[Thm 6.9]{toda12} has also shown that the genus 0 GV invariants can be extracted from the usual Donaldson-Thomas (DT) partition function. In particular when $X$ satisfies the conjecture of Maulik, Nekrasov, Okounkov, and Pandharipande \cite{mnop1}, then the genus 0 GV invariants and genus 0 Gromov-Witten invariants $N^0_\beta(X)$ satisfy the relation:
\begin{equation}
\label{eq:gw-relation}
N^0_\beta(X)=\sum_{k|\beta}{\frac{n^0_{\beta/k}(X)}{k^3}}.
\end{equation}
In practice, these GV invariants can be hard to compute, particularly when $X$ is compact, and have been computed explicitly in very few cases.

In this paper, we directly compute genus 0 GV invariants of certain fiber cohomology classes of curves on a compact Calabi-Yau threefold $X=\Xb$, see Theorem~\ref{th:main}. The result we obtain agrees with the above predictions using the DT invariants for this threefold recently computed in \cite{bryan19}. The generating function for the invariants is given by a Jacobi form.

\subsection{Definition of the Banana manifold $\Xb$}

Let $S$ be a generic rational elliptic surface. We view $S\subset \PP^1 \times \PP^2$ as a generic hypersurface of degree $(1,3)$. Then $S\rightarrow \PP^1$ is an elliptic fibration with 12 singular nodal fibers. The fiber product $S\times_{\PP^1} S$ is a singular threefold which has 12 conifold singularities. We describe the construction of the Banana manifold $\Xb$, and refer the reader to \cite{bryan19} for more details.

\begin{definition}\label{def:Xban} Given $S$ as above, we define the Banana manifold $\Xb$ to be
\[
\Xb = \Bl_{\Delta}(S\times_{\PP^1} S),
\]
the conifold resolution of the fiber product $S\times_{\PP^1} S$ given by blowing up along the diagonal $ \Delta\subset S\times_{\PP^1} S$.
 
The Banana manifold is a smooth compact Calabi-Yau threefold that has the structure of an Abelian surface fibration $\pi:\Xb\to \PP^1$ with exactly 12 singular fibers which are each isomorphic to a surface we call $\Fsing$. The surface $\Fsing$ is $\PP^1 \times \PP^1$ blown up at two points and glued along opposite edges,
\begin{equation}
\Fsing \cong \Bl_{\Delta}(\text{nodal curve }\times_{\PP^1} \text{ nodal curve})\subset \Xb.
\label{eq:fsing}
\end{equation}

Each singular fiber $\Fsing$ contains a curve that we call a Banana configuration, or Banana curve $\Cban$, Figure \ref{figure:bananacurve}. The Banana curve is a union of 3 rational curves intersecting in two points:
\begin{equation}
\Cban = C_1 \cup C_2 \cup C_3, \qquad C_i\cong \PP^1, \qquad C_i \cap C_j = \{p,q\}
\label{eq:cban}
\end{equation}
\[
N_{C_i/\Xb} = \OO(-1)\oplus \OO(-1).
\]
\begin{figure}[th!]
  \centering
	\scalebox{.5}{\includegraphics{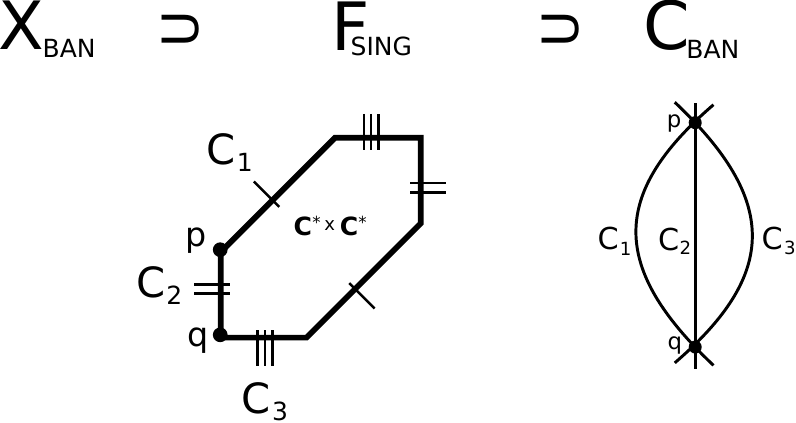}}
  \caption{A singular fiber $\Fsing$ containing the eponymous Banana curve $\Cban$.}
	\label{figure:bananacurve}
\end{figure}
The rational components $C_1$ and $C_2$ are the proper transforms of the nodal curves on respectively the first and second rational elliptic surfaces $S$ in the fiber product $S\times_{\PP^1} S$, while $C_3$ is the exceptional curve from the conifold resolution.

The singular locus of the map $\pi:\Xb\to \PP^1$ is the disjoint union of the twelve copies of $\Cban$, each of which lies on one of the twelve singular fibers isomorphic to $\Fsing$ of $\Xb$. We denote this collection of singular curves as
\[
\amalg\,\Cban\coloneqq \amalg_{i=1}^{12} (\Cban)_i,
\]
and the twelve singular fibers as 
\[
\amalg\,\Fsing\coloneqq \amalg_{i=1}^{12} (\Fsing)_i.
\]

The geometry of the fibration $\pi:\Xb\to \PP^1$ gives a group scheme structure to its smooth locus, which we call $\Xb^0$:
\begin{equation}
\Xb^0=\Xb \backslash \amalg\Cban \to \PP^1.
\label{eq:xbansmooth}
\end{equation}
This action extends to an action of $\Xb^0\to\PP^1$ on all of $\Xb\to\PP^1$, see \cite[\textsection 4.5]{bryan19}

Moreover, let $\Theta \subset H_2(\Xb,\ZZ)$ be the sublattice of fibers classes, namely classes represented by cocycles supported on a fiber. Then $\Theta$
 is spanned by $[C_1]$, $[C_2]$, and $[C_3]$: $$\Theta = \ZZ[C_1]\oplus \ZZ[C_2]\oplus \ZZ[C_3].$$
\end{definition}

\subsection{Main Result} 
Our main result is the following:
\begin{theorem} Let $\Xb$ be as above. Fix a curve class $\beta_{\mathbf{d}}$,
\[
\beta_{\mathbf{d}} = d_1[C_1]+d_2[C_2] + [C_3]\in H_2(\Xb), \qquad \mathbf{d}=(d_1, d_2)\in\ZZ^2_{\geq0}.
\]
The genus 0 Gopakumar-Vafa invariants $\GVban$ are determined by the following equation:
\begin{equation}
\label{eq:main}
\sum_{d_1, d_2} {\GVban x^{d_1}y^{d_2}} = 12 \prod_{m=1}^{\infty}{\frac{(1-x^my^{m-1})^2 (1-x^{m-1}y^{m})^2}{(1-x^my^m)^4}}.
\end{equation}
\label{th:main}
\end{theorem}
\begin{corollary}\label{cor:jacobi}
After the change of variables,
\[
q=xy,  \qquad p = y,
\]
the genus 0 GV invariants satisfy the identity:
\[
\sum_{d_1, d_2} {\GVban q^{d_1}p^{d_2-d_1-1} } = 12 \phi_{-2,1}(q,p),
\]
where $\phi_{-2,1}(q,p)$ is the unique weak Jacobi form of weight -2 and index 1:
\[
\phi_{-2,1}(q,p) = p^{-1}(1-p)^2\prod_{m=1}^{\infty}{\frac{(1-q^{m}p^{-1})^2 (1-q^{m}p)^2}{(1-q^m)^4}}.
\]
\[q=\exp(2\pi i\tau),  \qquad p = \exp(2\pi iz), \qquad  (\tau,z)\in \HH\times\CC.\]

\end{corollary}

In particular, this Jacobi form is one of the two generators of the ring of weak Jacobi forms. Furthermore, the index 1 weak Jacobi forms have a Fourier expansion $\sum{c(4n-r^2)q^np^r}$ whose coefficients $c(4n-r^2)$ depend only on a quadratic expression in the degrees \cite{eichler-zagier85}. We get the immediate consequence:
\begin{corollary}
The genus 0 GV invariants depend only on a quadratic function of the curve class. Namely, they satisfy $\GVban = n^0_{\lvert\lvert\beta_{\mathbf{d}}\rvert\rvert}(\Xb)$, where $\lvert\lvert{\beta_{\mathbf{d}}}\rvert\rvert\coloneqq 2d_1+2d_2+2d_1d_2-d_1^2-d_2^2-1.$ 
\end{corollary}

The appearance of the weak Jacobi form $\phi_{-2,1}(q,p)$ in our expression of the GV invariants is somewhat surprising and not well understood. This Jacobi form has appeared, for instance, in the DT partition function for certain elliptically fibered Calabi-Yau threefolds, as well as in other examples.   

\subsection{Outline of method} 
Our method of proof ultimately reduces the computation of the Behrend function weighted Euler characteristic of the moduli space $M=M^{\Xb}_{\beta_{\mathbf{d}}}$ to an actual count of structure sheaves of genus 0 curves. These curves are possibly nonreduced curves in the universal cover $\UFs$ of $\Fsing$. This is a sheaf theoretic analogue of the Gromov-Witten technique of passing to counts of genus 0 curves on the universal cover \cite{bryan-katz-leung01}.

The main idea behind the reductions is to use the motivic nature of weighted Euler characteristics. This allows us to compute using stratification and fixed point sets, even though we do not have a global $\CC^*$ action on our moduli space. An outline of the proof is as follows.

We begin in section \ref{sec: setup} by proving that a stable sheaf is scheme-theoretically supported on a single fiber. This gives us a map $M\to \PP^1$. It then suffices to compute the Behrend function weighted Euler characteristic $e(M,\nu)$ fiberwise. The group scheme action of $\Xb^0\to\PP^1$ on $\Xb\to\PP^1$ induces a fiberwise action on the moduli space, where the group of each fiber of $\Xb^0\to\PP^1$ acts on the corresponding fiber of $M\to \PP^1$. This fiberwise group action preserves the symmetric obstruction theory of the moduli space and hence preserves $\nu$. Thus, $e(M,\nu)$ can be computed on orbits of this action.
 
The generic smooth fibers of $\Xb$ are non-singular Abelian surfaces where the group action is transitive and support no invariant curves. Consequently, the sheaves supported on the smooth fibers contribute zero to $e(M,\nu)$. On the singular fibers $\Fsing$, the group action gives a natural $\CC^*\times \CC^*$ torus action on the moduli space. The fixed points of this action are the only stable sheaves that contribute to $e(M,\nu)$. These sheaves are scheme-theoretically supported on the singular fibers $\Fsing$ with set-theoretic support on the Banana curve configuration $\Cban$. Thus we reduce the problem of computing $\GVban$ to that of counting torus-invariant stable sheaves on $\Fsing$ of curve class $\beta_{\mathbf{d}}$ and Euler characteristic 1.

This count corresponds to the naive Euler characteristic of our moduli space, $\naivev$. This is defined as the Euler characteristic without the Behrend function weighting:
\[
\naivev \coloneqq e(M^{\Xb}_{\beta_{\mathbf{d}}}).
\]
We begin by determining these.

We show in section \ref{sec: equiv} that, in fact, it suffices to count those invariant stable sheaves on $\Fsing$ that push forward from the universal cover $\UFs$. This involves considering the action on the moduli space given by tensoring by line bundles on $\Fsing$. Any sheaf fixed under this action must pull back to an equivariant sheaf on $U(\Fsing)$ which contains a distinguished subsheaf isomorphic under pushforward to the original. Now we need to determine how many of these distinguished stable torus-invariant sheaves there are on $U(\Fsing)$.

These distinguished sheaves on $U(\Fsing)$ can be counted using a combinatorial argument detailed in sections \ref{sec: count} and \ref{sec:combinatorics}. The assumption of Euler characteristic 1 is very restrictive, and together with some elementary stability arguments, we show that the only torus invariant stable sheaves that push forward to invariant sheaves in our moduli space are structure sheaves of arithmetic genus 0 curves that satisfy certain constraints on adjoining components. Such curves can be classified by combinatorics in terms of the number of integer partitions whose odd parts are distinct, and we obtain a closed form generating function for $\naivev$.

Finally, in section \ref{sec:behrend function}, we prove that $\GVban$ is related to $\naivev$ by a sign change. We use the result of \cite[Corollary 3.5]{behrend-fantechi08} that given a $\CC^*$ action with isolated fixed points $[\FF]\in M^{\CC^*}$, the weighted Euler characteristic depends only on the parities of the dimension of the tangent spaces at those points:  
\[
e(M,\nu) = \sum_{[\FF]\in M^{\CC^*}}{(-1)^{\dim T_{[\FF]}M}}.
\]
In \cite{behrend-fantechi08}, this comes from the computation of the weighted Euler characteristic of the Milnor fiber in the presence of a $\CC^*$ action. However, the dimension of the tangent space at isolated fixed points may also be computed using virtual localization, as in \cite{mnop1}. We use the formula given in \cite{mnop1} to calculate the dimension of the groups $\Ext^1(\FF, \FF)$ for our fixed points $[\FF]$, and thus the parity of $\dim(T_{[\FF]}M)$. This finishes the proof of our main result.

Our method is limited to curve classes $d_1[C_1]+d_2[C_2] + [C_3]$ since a simple count of combinations of structure sheaves does not appear to suffice in the general case. For $d_1[C_1]+d_2[C_2] + d_3[C_3]$ with $d_1, d_2, d_3>1$, there are corresponding sheaves on the universal cover which are stable with Euler characteristic 1, but are not structure sheaves. At present, we do not know how to analyze the moduli space in these cases.

%%%%%%%%%%%%%%%%%%%%%%%%%%%%%%%
%\input{./section1}
\section{Setup to counting on $\Fsing$} \label{sec: setup}

Throughout the rest of this section, we let $X=\Xb \xrightarrow{\pi}\PP^1$ and $M=M^X_{\beta}$, as given by Definitions \ref{def:MXbeta} and \ref{def:Xban} in the introduction.

We begin with two observations that hold for Simpson semistable pure 1-dimensional sheaves $\FF$ with $\chi(\FF) = 1$. First, all semistable sheaves are in fact stable. Second, the stability condition can be restated in terms of Euler characteristic of subsheaves or quotient sheaves.

Recall the definition of semistable and stable sheaves \cite{huybrechts-lehn10}.
\begin{definition} Let $Y$ be a complex projective scheme and $\FF$ a pure coherent sheaf of dimension $d$ on $Y$. Fix an ample line bundle $H=\OO(1)$ on $Y$.
Define the Hilbert polynomial $P(\FF,m)$ and reduced Hilbert polynomial $p(\FF,m)$ of $\FF$ as follows:
\begin{align*}
P(\FF,m) &\coloneqq\chi(\FF\otimes \OO(m))= \sum\limits_{i=0}^{d}{\frac{\alpha_i(\FF)}{i!}m^i},\\
p(\FF,m) &\coloneqq \frac{P(\FF,m)}{\alpha_d(\FF)}.
\end{align*}
We say $\FF$ is semistable if for any proper subsheaf $\FF'\subset \FF$, $p(\FF') \leq p(\FF)$. A semistable sheaf is called stable if the inequality is strict.
\end{definition}

\begin{lemma}
There are no strictly semistable sheaves in $M$. Moreover, the stability condition for $\EE\in M$ is equivalent to the following: $\EE$ is stable if and only if $\chi(\EE')\leq 0$ for any proper subsheaf $\EE'\hookrightarrow \EE$. Equivalently, $\EE$ is stable if and only if $\chi(\EE'')>0$ for any quotient sheaf $\EE\twoheadrightarrow \EE''\neq 0$.
\label{lemma: nosemistable-euler}
\end{lemma}

\begin{proof}

This follows from the definition of stability.
\end{proof}

\begin{corollary}
$M$ is independent of the choice of polarization $H$.
\end{corollary}

We begin by showing that the moduli space $M$ has the structure of a scheme over $\PP^1$.

\begin{proposition}\label{prop:fibersupp}
Suppose $\beta$ is a curve class such that $\pi_*\beta=0$. Let $\EE\in M$. Then $\EE$ is scheme theoretically supported on a single fiber.
\end{proposition}

\begin{proof} Let $C=(\Supp\EE)_{\textsc{red}}$ be the reduced support of $\EE$. Since $\beta$ is a fiber class, $\Supp\pi_*\EE = \{p_i\}$ is a finite set of points, so $C$ is a collection of fibers. But direct sums are necessarily unstable so the support of $\EE$ must be connected. Hence, the set theoretic support of  $\EE$ is contained in a single fiber $F=F_x$, for some $x\in\PP^1$.

Now $i:F\hookrightarrow X$ is a closed subscheme so we have the exact sequence:
$$0\rightarrow \II_{F/X}\rightarrow \OO_X\rightarrow i_{*}\OO_F\rightarrow 0.$$ 

Since $F$ is an effective Cartier divisor on the nonsingular $X$, $\OO_X(F)$ is locally free and we can tensor by $\OO_X(F)$ to get the short exact sequence:
$$0\rightarrow \OO_X \rightarrow \OO_X(F)\rightarrow i_{*}\OO_F(F)\rightarrow 0.$$

The normal bundle of the fiber class $F$ is trivial, so $\OO_F(F)\cong\OO_F$ and we get:
$$0\rightarrow \OO_X \rightarrow \OO_X(F)\rightarrow i_{*}\OO_F\rightarrow 0,$$
which we can tensor with $\EE$,
$$ \EE \rightarrow \EE(F)\rightarrow \EE_F\rightarrow 0.$$

Again, because $\EE$ is supported on the fiber class $F$, and $\OO_X(F)$ is a trivial line bundle when restricted to $F$, $\EE(F)\cong\EE$:
$$ \EE \rightarrow \EE\rightarrow \EE_F\rightarrow 0.$$

By stability, $\EE \rightarrow \EE$ is either the zero map or an isomorphism, which implies $\EE_F=0$ or $\EE\cong \EE_F$. By assumption, $\EE$ and hence $\EE_F$ is nonzero, so $\EE \cong \EE_F$ and $\EE$ is scheme-theoretically supported on $F$.
\end{proof}

This gives us a natural map $\rho:M\to \PP^1$ which allows us to compute the \wec\ fiberwise. Recall that for any constructible morphism, the weighted Euler characteristic can be computed as a pushforward \cite{macpherson74}. So the map $\rho:M\to \PP^1$ allows us to compute the Behrend function weighted Euler characteristics fiberwise. Thus, we get:
\[
e(M,\nu)=e(\PP^1, \rho_*\nu),
\]
where $(\rho_*\nu)(t)= e(M_t, \nu_t)$ for $t\in \PP^1$, $M_t\coloneqq\rho^{-1}(t)$, and $\nu_t\coloneqq\nu|_{M_t}$.

Recall (Eq.~\ref{eq:xbansmooth}) that $\Xb^0\eqqcolon X^0$, the smooth locus of the fiber map of the Banana manifold is a group scheme $X^0\to\PP^1$ that acts on the Banana manifold $\pi:X\to\PP^1$.

Let $X_t$ be a fiber of $X$, and $G_t$ the group of the fiber:
\begin{align*}
X_t&\coloneqq \pi^{-1}(t)\subset X,\\
G_t&\coloneqq X^0\cap X_t, \qquad t\in \PP^1.
\end{align*}

The nonsingular fibers are Abelian surfaces which are products of an elliptic curve $E$ with itself. In this case, the group of the fiber is the fiber itself, 
\begin{align*}
G_t&=X_t, \quad\text{ when } X_t =E\times E,
\end{align*}
and $G_t$ acts by translation in the group law.

On the singular fibers $\Fsing$ (Eq.~\ref{eq:fsing}), the group of the fiber is a torus acting by translation,  
\begin{align*}
G_t&=\Fsing\backslash \Cban \cong \CC^*\times \CC^*, \quad\text{ when } X_t =\Fsing.
\end{align*}

\begin{proposition} \label{prop:tfixcount}
To compute $n^0_\beta(X)$, it suffices to count those sheaves of $M$ with scheme theoretic support contained in $\amalg\,\Fsing$ and set-theoretic support contained in $\amalg\,\Cban$, which are also invariant under the action of the group scheme $X^0\to\PP^1$. 
\end{proposition}
\begin{proof}
On each fiber, $X_t=\pi^{-1}(t)\subset X, t\in \PP^1$, the group of the fiber $G_t$ also acts on sheaves supported on $X_t$, which in turn, induces an action of $G_t$ on $M_t= \rho^{-1}(t)$. 
We will show in Section~\ref{sec:behrend function} that the group scheme action is trivial on $K_X$ and preserves the Behrend function $\nu$. This algebraic group action of $G_t$ on $M_t$, gives us a stratification of $M_t$ into locally closed equivariant subsets. By \cite{behrend09, behrend-fantechi08}, $e(M_t, \nu_t)$ can be computed on orbits of this action.

In particular, if the topological Euler characteristic of the group vanishes, $e(G_t)=0$, as it does here, then $e(M_t)= e(M_t^{G_t})$, because the Euler characteristic can be computed by strata. The fixed points of the group action on the moduli space corresponds to an isomorphism class of sheaves, $[E]$ such that $[E]\cong[g^* E]$, where $g:X_t\to X_t$ is the action on the underlying space given by the group element $g\in G_t$. In particular, the support of the sheaf has to be preserved by the group action. 
Over general points $t\in\PP^1$, the fiber $X_t$ is smooth and the group action is that of the Abelian surface acting transitively on itself through translation in the group law. So these fibers contain no invariant curves. Consequently, the sheaves supported on smooth fibers do not contribute to $e(M, \nu)$. 

On the singular fibers isomorphic to $\Fsing$, the Banana curve $\Cban$ is the only curve preserved by the action of $\CC^*\times \CC^*$. Thus we have reduced our problem of computing $n^0_\beta(X)$ to counting only those sheaves $\FF\in M$ which are $\CC^*\times \CC^*$-invariant and with $(\Supp\FF)_{\textsc{red}}\subset\amalg\,\Cban$. By Proposition \ref{prop:fibersupp}, these sheaves are scheme-theoretically supported on $\amalg\,\Fsing$.
\end{proof}

Since each of the twelve singular fibers are isomorphic to, and disjoint from, each other, it suffices to count the torus-invariant sheaves in $M$ supported on only one of these fibers. Multiplying this count by twelve then gives the invariant $n^0_{\beta}(X)$. For the remainder of the paper, we will focus on such sheaves supported on one of the singular fibers.

\begin{definition}
Fix one of the singular fibers of $X$, which we will also call $\Fsing$.

Define $\Msing\subset M$ to be sheaves in $M$ supported on $\Fsing$,
\[
\Msing = \{[\FF]\in M| \Supp \FF \subset \Fsing\} \subset M.
\]

Let $T$ be the 2-torus which acts on the fiber $\Fsing$ and thus on $\Msing$:
%\[
\begin{gather*}
T\coloneqq X^0\cap \Fsing \cong \CC^*\times \CC^*,\\
T \text{ acts on }\Msing.
\end{gather*}
%\] 

Define $\MsingT$ to be sheaves in $\Msing$ invariant under the action of $T$:
\[
\MsingT \coloneqq \left\{[\FF]\in \Msing |\, [ \FF ] \text{ invariant under } T \right\}.
\]
\end{definition}

With this notation in place, the following corollary to Proposition \ref{prop:tfixcount} is immediate.
\begin{corollary} The Gopakumar-Vafa invariants of $X$ can be computed from the Behrend function weighted count of $\MsingT$:
\[
e(M,\nu) = 12 e(\MsingT,\nu|_{\MsingT}).
\]
\end{corollary}

%%%%%%%%%%%%%%%%%%%%%%%%%%%%%%%
%\input{./section2}
\section{Geometry}
We want to convert our problem into one of counting sheaves on the universal cover $\UFs\xrightarrow{\pr}\Fsing$. In this section, we explain some of its geometry that we will need in the rest of the paper.
\subsection{Geometry of $\UFs$}\label{geometry}
First we discuss some of the geometry of the universal cover, although we will not need the description of the formal neighborhood until Section~\ref{sec:behrend function}.

\begin{notation} Denote by
\begin{itemize}
\item $\Fhat$: the formal completion of $X$ along $\Fsing$, 
\item $\UFs$: the universal cover of the singular fiber $\Fsing$, 
\item $\UFhat$ : the universal cover of $\Fhat$,
\item $\textsc{Nrm}(\Fsing)$: the normalization of $\Fsing$, 
\item $\Fnormwidehat$: the formal completion of the total space of the canonical bundle of the blow up of $\PP^1 \times \PP^1$ at the two torus fixed antidiagonal points, along the zero section, 
\begin{equation}
\begin{aligned}
\Fnormwidehat\cong \reallywidehat{\Bl_{a,b}(\PP^1 \times \PP^1)} &\lhook\joinrel\longrightarrow \textsc{Tot }K(\Bl_{a,b}\PP^1 \times \PP^1),\\
 &\{a,b\}=\{(0,\infty), (\infty,0)\} \in \PP^1 \times \PP^1.
\end{aligned}
\end{equation}
\end{itemize}
\end{notation}

We regard $\Fhat$ as a formal Calabi-Yau threefold. In \cite[Proposition 4.10]{bryan19} it is shown that $\Fnormwidehat$ is an \'{e}tale cover of $\Fhat$,
\[
 \reallywidehat{\textsc{Nrm}(\Fsing)} \xrightarrow{\text{\'{e}tale}}  \Fhat.
\]

The momentum polytope of $\textsc{Nrm}(\Fsing)$ and its toric fan are pictured in Figures \ref{fig:mompoly_normfsing} and \ref{fig:fan_normfsing}. 

\begin{figure}[th!]
  \centering
  \scalebox{.5}{
\begin{tikzpicture}[y=0.80pt, x=0.80pt, yscale=-1.000000, xscale=1.000000, inner sep=0pt, outer sep=0pt]
\begin{scope}[shift={(0,0)}]
  \path[fill=black,line join=miter,line cap=butt,line width=0.800pt]
    (0.0000,0.0000) node[above right] (flowRoot4196) {};
  \path[draw=black,line join=miter,line cap=butt,even odd rule,line width=1.378pt]
    (5.3335,143.7407) .. controls (5.3335,75.6406) and (5.3335,75.6406) ..
    (5.3335,75.6406);
  \path[draw=black,line join=miter,line cap=butt,even odd rule,line width=1.378pt]
    (5.4850,75.7833) -- (74.2807,6.9876);
  \path[draw=black,line join=miter,line cap=butt,even odd rule,line width=1.378pt]
    (142.6367,75.2921) .. controls (142.6367,7.1920) and (142.6367,7.1920) ..
    (142.6367,7.1920);
  \path[draw=black,line join=miter,line cap=butt,even odd rule,line width=1.378pt]
    (74.0153,7.2995) .. controls (142.1154,7.2995) and (142.1154,7.2995) ..
    (142.1154,7.2995);
  \path[draw=black,line join=miter,line cap=butt,even odd rule,line width=1.378pt]
    (5.8776,142.9363) .. controls (73.9777,142.9363) and (73.9777,142.9363) ..
    (73.9777,142.9363);
  \path[draw=black,line join=miter,line cap=butt,even odd rule,line width=1.378pt]
    (74.0028,143.5539) -- (142.7984,74.7582);
\end{scope}

\end{tikzpicture}
	}
  \caption{Momentum polytope of $\textsc{Nrm}(\Fsing)$}
	\label{fig:mompoly_normfsing}
\end{figure}
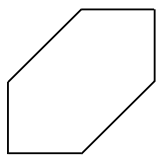

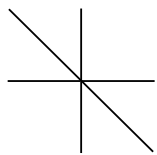
\begin{figure}[th!]
  \centering
  \scalebox{.5}{
	\begin{tikzpicture}[y=0.80pt, x=0.80pt, yscale=-1.000000, xscale=1.000000, inner sep=0pt, outer sep=0pt]
\begin{scope}[shift={(0,0)}]
  \path[fill=black,line join=miter,line cap=butt,line width=0.800pt]
    (0.0000,0.0000) node[above right] (flowRoot4196) {};
  \path[draw=black,line join=miter,line cap=butt,even odd rule,line width=1.378pt]
    (4.8805,75.3092) .. controls (72.9806,75.3092) and (72.9806,75.3092) ..
    (72.9806,75.3092);
  \path[draw=black,line join=miter,line cap=butt,miter limit=4.00,even odd
    rule,line width=1.378pt] (6.0683,7.5708) -- (141.0307,141.9903);
  \path[draw=black,line join=miter,line cap=butt,even odd rule,line width=1.378pt]
    (74.3441,75.3092) .. controls (142.4442,75.3092) and (142.4442,75.3092) ..
    (142.4442,75.3092);
  \path[draw=black,line join=miter,line cap=butt,even odd rule,line width=1.378pt]
    (73.6737,6.9161) .. controls (73.6737,75.0162) and (73.6737,75.0162) ..
    (73.6737,75.0162);
  \path[draw=black,line join=miter,line cap=butt,even odd rule,line width=1.378pt]
    (73.6737,75.0972) .. controls (73.6737,143.1974) and (73.6737,143.1974) ..
    (73.6737,143.1974);
\end{scope}

\end{tikzpicture}

	}
  \caption{Toric fan of $\textsc{Nrm}(\Fsing)$}
	\label{fig:fan_normfsing}
\end{figure}

Then $\Fnormwidehat$, the formal neighborhood of the normalization of the singular fiber, is formally locally isomorphic to the total space of the canonical bundle of the blow up of $\PP^1 \times \PP^1$ at two points, which is the toric three-fold associated to the fan depicted in Figure \ref{fig:fan_totk}. This fan comes from constructing cones over the two-dimensional polytope of Figure \ref{fig:fan_normfsing} placed at height 1 in $\RR^3$.

\begin{figure}[th!]
  \centering
	\scalebox{.5}{
	\begin{tikzpicture}[y=0.80pt, x=0.80pt, yscale=-1.000000, xscale=1.000000, inner sep=0pt, outer sep=0pt]
\begin{scope}[shift={(0,0)}]
  \path[fill=black,line join=miter,line cap=butt,line width=0.800pt]
    (0.0000,0.0000) node[above right] (flowRoot4196) {};
    \path[draw=black,line join=miter,line cap=butt,miter limit=4.00,even odd
      rule,line width=1.378pt] (11.0655,82.3139) -- (308.9403,34.5123);
    \path[draw=black,line join=miter,line cap=butt,even odd rule,line width=0.800pt]
      (70.0756,104.5401) -- (290.9213,59.8254);
    \path[draw=black,line join=miter,line cap=butt,miter limit=4.00,even odd
      rule,line width=0.800pt] (28.2491,58.9869) -- (236.0860,22.1375);
    \path[draw=black,line join=miter,line cap=butt,even odd rule,line width=0.800pt]
      (290.8178,58.5862) -- (307.0798,34.1354);
    \path[draw=black,line join=miter,line cap=butt,even odd rule,line width=0.800pt]
      (238.8510,21.1767) -- (305.4013,34.4365);
    \path[draw=black,line join=miter,line cap=butt,even odd rule,line width=0.800pt]
      (12.2980,83.4703) -- (70.9210,103.5506);
    \path[draw=black,line join=miter,line cap=butt,even odd rule,line width=0.800pt]
      (11.0898,83.4353) -- (29.6555,57.4487);
    \path[draw=black,line join=miter,line cap=butt,even odd rule,line width=0.800pt]
      (160.0000,2.3622) -- (160.0000,202.3622);
    \path[draw=black,line join=miter,line cap=butt,even odd rule,line width=0.800pt]
      (10.0000,152.3622) -- (305.0000,152.3622);
    \path[draw=black,line join=miter,line cap=butt,miter limit=4.00,even odd
      rule,line width=0.800pt] (95.4564,182.9114) .. controls (217.3835,122.7206)
      and (217.3835,122.7206) .. (217.3835,122.7206);
    \path[draw=black,line join=miter,line cap=butt,miter limit=4.00,even odd
      rule,line width=1.378pt] (27.1140,59.3749) -- (290.6287,59.3749);
    \path[draw=black,line join=miter,line cap=butt,miter limit=4.00,even odd
      rule,line width=1.378pt] (71.4842,102.9978) -- (240.2481,20.9164);
    \path[draw=black,line join=miter,line cap=butt,even odd rule,line width=0.800pt]
      (70.0000,102.3622) -- (160.0000,152.3622) -- (290.7679,59.2905) --
      (290.0000,58.9065);
    \path[draw=black,line join=miter,line cap=butt,even odd rule,line width=0.800pt]
      (10.0000,82.3622) -- (160.0000,152.3622);
    \path[draw=black,dash pattern=on 0.80pt off 1.60pt,line join=miter,line
      cap=butt,miter limit=4.00,even odd rule,line width=0.800pt] (310.0000,32.3622)
      -- (160.0000,152.3622) -- (160.0000,152.3622);
    \path[draw=black,dash pattern=on 0.80pt off 1.60pt,line join=miter,line
      cap=butt,miter limit=4.00,even odd rule,line width=0.800pt]
      (160.0000,152.3622) -- (27.6962,58.8981);
    \path[draw=black,dash pattern=on 0.80pt off 1.60pt,line join=miter,line
      cap=butt,miter limit=4.00,even odd rule,line width=0.800pt]
      (160.0000,152.3622) -- (240.0000,22.3622);
\end{scope}

\end{tikzpicture}
	}
 \caption{Toric fan of $\textsc{Tot }K(\Bl_{a,b}(\PP^1\times\PP^1)) \xleftarrow{\text{formal}}\joinrel\rhook  \widehat{\textsc{Nrm}(\Fsing)} $}
	\label{fig:fan_totk}
\end{figure}

The map $\textsc{Nrm}(\Fsing) \to \Fsing$ can be described by identifying opposite edges of the momentum polytope. In the case of $\Fnormwidehat\xrightarrow{\text{\'{e}tale}} \Fhat$, the gluing is done along a formal open neighborhood of the edges of the polytope. 

From this geometry, we see that the $\UFs$ has a piecewise smooth map to $\RR^2$. On each component, this is a moment map for the $T\cong \CC^*\times \CC^*$ action with hexagonal momentum polytope. The image of this map has a planar projection given by an infinite tiling with hexagons, as shown in Figure \ref{fig:hex_tile}. The cones over the dual tiling in Figure \ref{fig:triangle_tile}, placed at height 1, is then the fan associated to a non-finite type toric Calabi-Yau three-fold $\mathfrak{W}$ Figure \ref{fig:triangle_cone} to which $\UFhat$ is formally locally isomorphic. We will return to this viewpoint in Section~\ref{sec:behrend function}.
\begin{figure}[th!]
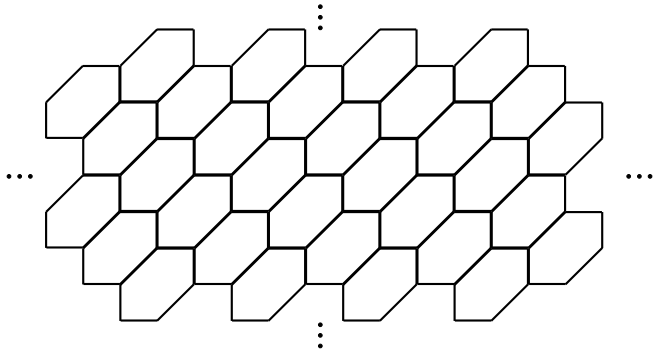

  \centering
	\scalebox{.5}{
	


	}
  \caption{Momentum polytope of $U(\Fsing)$}
	\label{fig:hex_tile}
\end{figure}
\begin{figure}[th!]
  \centering
	\scalebox{.4}{
	%

	}
  \caption{Dual tiling of $U(\Fsing)$}
	\label{fig:triangle_tile}
\end{figure}
\begin{figure}[th!]
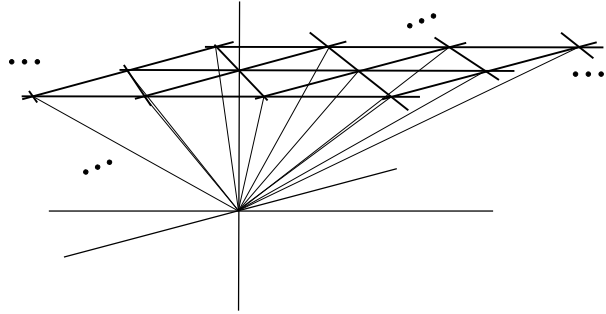

  \centering
	\scalebox{.6}{
	%

	}
  \caption{Non-finite type toric Calabi-Yau three-fold $\mathfrak{W}$ formally locally isomorphic to $\UFhat$}
	\label{fig:triangle_cone}
\end{figure}
From this description in terms of the momentum polytope, we see that the group of deck transformation of $\UFs$ is free abelian on two generators, $\pi_1(\Fsing)\cong \ZZ\times\ZZ$. 

\subsection{Local geometry of $\Cban$}\label{localbanana}
As a consequence of Proposition~\ref{prop:tfixcount}, we only need to consider sheaves with scheme-theoretic support on $\Fsing$. We will thus restrict our discussion in the following sections to studying sheaves on the surface $\Fsing$. In particular, the support of such sheaves can only have thickenings in this surface, and not more generally in other $\Xb$ directions.

The geometry of $\Cban \subset\Fsing$ is such that it looks formally locally like the total space of the union of the two $\CC^*\times\CC^*$-invariant subbundles of $\OO(-1) \oplus\OO(-1)\to \PP^1$, see Figure~\ref{fig:local p1}. 

\begin{figure}[th!]
  \centering
		\scalebox{2}{
		
\begin{tikzpicture}[y=0.80pt, x=0.80pt, yscale=-1.000000, xscale=1.000000, inner sep=0pt, outer sep=0pt]
\begin{scope}[shift={(0,0)}]
	\node at (36,35) {\scalebox{.7}{$\scriptstyle C_3$}};
	\node at (56,23)  {\scalebox{.7}{$\scriptstyle q$}};
	\node at (24,57) {\scalebox{.7}{$\scriptstyle p$}};
	
	\node at (60,10) {\scalebox{.7}{$\scriptstyle C_1$}};
	\node at (35,72) {\scalebox{.7}{$\scriptstyle C_1$}};
	
	\node at (74,33) {\scalebox{.7}{$\scriptstyle C_2$}};
	\node at (10,58) {\scalebox{.7}{$\scriptstyle C_2$}};

  \path[fill=black,line join=miter,line cap=butt,line width=0.800pt]
    (0.0000,0.0000) node[above right] (flowRoot4196) {};
  \begin{scope}[shift={(-187.41261,-25.65353)}]
    \path[draw=black,line join=miter,line cap=butt,even odd rule,line width=0.800pt]
      (189.7821,77.7947) -- (190.0000,77.7363) -- (215.0000,77.7363) .. controls
      (215.0000,102.3622) and (215.0000,102.3622) .. (215.0000,102.3622);
    \path[draw=black,line join=miter,line cap=butt,even odd rule,line width=0.800pt]
      (214.8963,77.7947) .. controls (240.0000,52.6910) and (240.0000,52.6910) ..
      (240.0000,52.6910);
    \path[draw=black,line join=miter,line cap=butt,even odd rule,line width=0.800pt]
      (265.3478,52.2783) -- (265.1298,52.3367) -- (240.1298,52.3367) .. controls
      (240.1298,27.7108) and (240.1298,27.7108) .. (240.1298,27.7108);
    \path[draw=black,line join=miter,line cap=butt,even odd rule,line width=0.800pt]
      (260.6919,50.2970) -- (260.6919,55.2970) -- (260.6919,55.2970);
    \path[draw=black,line join=miter,line cap=butt,even odd rule,line width=0.800pt]
      (195.6426,75.2505) -- (195.6426,80.2505) -- (195.6426,80.2505);
    \begin{scope}[cm={{0.0,1.0,-1.0,0.0,(330.1928,-49.13237)}}]
      \path[draw=black,line join=miter,line cap=butt,even odd rule,line width=0.800pt]
        (82.5793,87.7587) -- (82.5793,92.7587) -- (82.5793,92.7587);
      \begin{scope}[shift={(0,0)},draw=black,fill=black]
        \path[draw=black,line join=miter,line cap=butt,even odd rule,line width=0.800pt]
          (82.5793,87.7587) -- (82.5793,92.7587) -- (82.5793,92.7587);
      \end{scope}
      \begin{scope}[shift={(2.17204,0)},shift={(0,0)},draw=black,fill=black]
        \path[draw=black,line join=miter,line cap=butt,even odd rule,line width=0.800pt]
          (82.5793,87.7587) -- (82.5793,92.7587) -- (82.5793,92.7587);
      \end{scope}
    \end{scope}
    \begin{scope}[cm={{0.0,1.0,-1.0,0.0,(305.57459,12.88107)}}]
      \path[draw=black,line join=miter,line cap=butt,even odd rule,line width=0.800pt]
        (82.5793,87.7587) -- (82.5793,92.7587) -- (82.5793,92.7587);
      \begin{scope}[shift={(0,0)},draw=black,fill=black]
        \path[draw=black,line join=miter,line cap=butt,even odd rule,line width=0.800pt]
          (82.5793,87.7587) -- (82.5793,92.7587) -- (82.5793,92.7587);
      \end{scope}
      \begin{scope}[shift={(2.17204,0)},shift={(0,0)},draw=black,fill=black]
        \path[draw=black,line join=miter,line cap=butt,even odd rule,line width=0.800pt]
          (82.5793,87.7587) -- (82.5793,92.7587) -- (82.5793,92.7587);
      \end{scope}
    \end{scope}
    \path[cm={{0.70711,-0.70711,0.70711,0.70711,(0.0,0.0)}},draw=black,fill=black,line
      join=round,line cap=rect,miter limit=4.00,line width=0.389pt]
      (97.0492,206.9839) circle (0.0343cm);
    \path[cm={{0.70711,-0.70711,0.70711,0.70711,(0.0,0.0)}},draw=black,fill=black,line
      join=round,line cap=rect,miter limit=4.00,line width=0.389pt]
      (133.2380,206.8879) circle (0.0343cm);
  \end{scope}
\end{scope}

\end{tikzpicture}

		}
  \caption{Local geometry of $\Cban$}
	\label{fig:local p1}
\end{figure}

Let $e\subset\UFs$ be any possibly nonreduced, irreducible curve which is a lift of a multiple of a $\Cban$ curve component, 
\[
[pr_*e] = d[C_i].
\] 
We will call such an $e$ an edge. Then $e$ is the intersection of two irreducible surface components in $\UFs$, that is, the intersection of two hexagons in the momentum polytope (Figure~\ref{fig:hex_tile}). Such an edge $e$ has possible thickenings that is determined by two numbers $m_e, n_e \geq 1$, which record the thickening in each of the two surface directions. More concretely, we can represent $\UFs$ near $e$ in local coordinates as the union of the $xy$ and the $xz$ coordinate planes in $\CC^3=\{(x,y,z) \}$. The $x$-axis where these planes meet then corresponds to the edge $e$. Since $e$ must lie in the surface $\UFs$, it has possible thickenings only in the $xy$ or $xz$ plane directions, and we can encode this information with two positive integers $m_e$ and $n_e$:
\[
e\cong \Spec\CC[x,y,z]/(y^{m_e}, z^{n_e}, yz).
\]

More generally, suppose $\CCC\subset\UFs$ is a curve that lies over $\Cban$, $(pr(\CCC))_{\textsc{red}}\subset\Cban$. We will call any point in $\CCC$ which is a preimage of a node of $\Cban$ a vertex. Such a vertex is a point of intersection of up to three edges. These edges project to the three components $C_i$ of $\Cban$ under the covering map. 

In an affine neighborhood of a vertex with three edges, we can express $\CCC$ in local coordinates with the vertex as the origin of $\CC^3$ and the edges as the coordinate axes. Then $\CCC$ has the structure: 
\[\Spec\CC[x,y,z]/(xyz, x^{r}y^{m}, z^{n}x^{b}, y^{a}z^{s})
\]
for some finite thickenings $a,b,m,n,r,s\geq 1$ as shown in Figure~\ref{fig:pile of boxes}. If there are only two components meeting at a vertex, say with a missing $z$-axis, then locally $\CCC$ is given by:
\[
 \Spec\CC[x,y,z]/(xyz, x^{r}y^{m}, z^{n}x, yz^{s}, z^{\max(n,s)}).
\]
This is in fact the same as the degree 3 vertex case if we take the convention that the empty edge has thickenings of lengths $0$ and $1$, where the $0$ length is taken in the direction of the axis with the larger thickening in their shared plane. For example, if $ n\geq s$,
\[
(xyz, x^{r}y^{m}, z^{n}x, yz^{s}, z^{\max(n,s)}) = (xyz, x^{r}y^{m}, z^{n}, yz^{s}).
\]

\begin{figure}[th!]
  \centering
	\scalebox{1.5}{
	%made from pileofboxes5 in inkscape saved to png,
%then cropped, then put into this file and manually added labels from pileofboxes4a

\begin{tikzpicture}[y=0.80pt, x=0.80pt, yscale=-1.000000, xscale=1.000000, inner sep=0pt, outer sep=0pt]

\begin{scope}[shift={(0,0)}]
   \node at (-20,0) {\includegraphics[scale=.6]{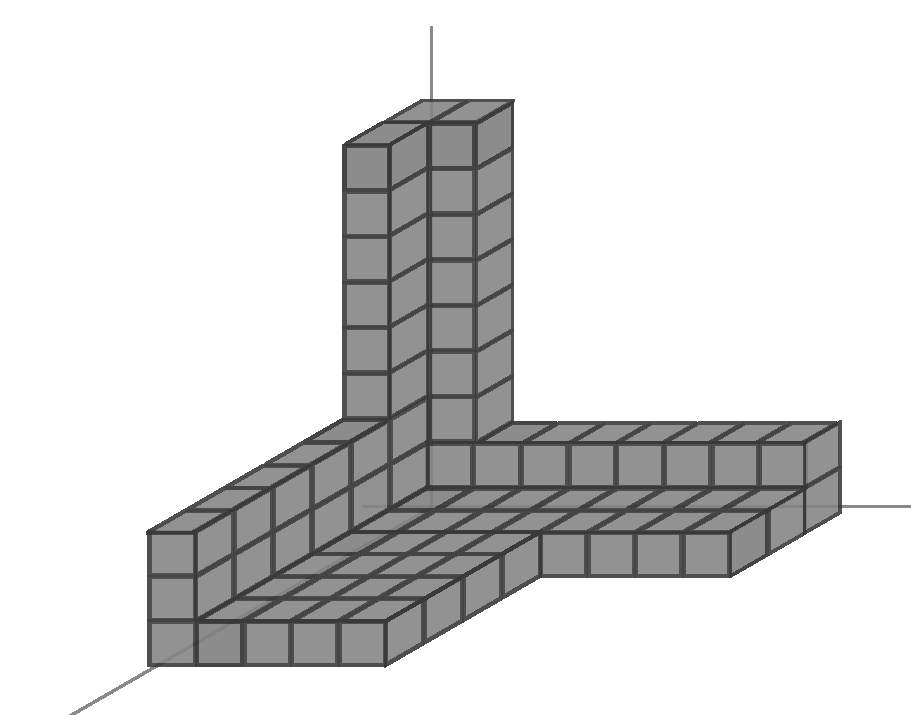}};
	
	\node at (-117,98) {$\scriptstyle x$};
	\node at (107,45)  {$\scriptstyle y\quad$};
	\node at (-32,-88) {$\scriptstyle z$};
	
	\node at (-110,65) {\resizebox{8pt}{18pt}{$\{$}};
	\node at (-120,65) {$\scriptstyle n$};
	\node at (-72,90) {\rotatebox{90}{\resizebox{8pt}{38pt}{$\{$}}};
	\node at (-72,99) {$\scriptstyle m$};
	\node at (-72,104) {$\scriptstyle \quad$};%add some bottom margin

	\node at (74,55) {\rotatebox{125}{\resizebox{8pt}{20pt}{$\{$}}};
	\node at (80,58) {{$\scriptstyle r$}};

	\node at (90,29) {\resizebox{8pt}{12pt}{$\}$}};
	\node at (97,29) {$\scriptstyle s$};

	\node at (-15,-76) {\rotatebox{270}{\resizebox{8pt}{12pt}{$\{$}}};
	\node at (-15,-83) {{$\scriptstyle a$}};

	\node at (-42,-68) {\rotatebox{300}{\resizebox{8pt}{15pt}{$\{$}}};
	\node at (-48,-75) {$\scriptstyle b$};

	\end{scope}
\end{tikzpicture}
	}
  \caption{Multiple structure of $\CCC$ near a vertex.}
	\label{fig:pile of boxes}
\end{figure}

%%%%%%%%%%%%%%%%%%%%%%%%%%%%%%%
%\input{./section3}
\section{Reduction to counting on $\UFs$} \label{sec: equiv}
In this section, we explain how to convert our problem into one of counting sheaves on the universal cover $\UFs\xrightarrow{\pr}\Fsing$. We do this by considering a second $\CC^*\times\CC^*$ action on $\MsingT$, that of tensoring with degree zero line bundles on $\Fsing$. 

Recall the categorical equivalence between equivariant sheaves on a covering space with a free group action and sheaves on the quotient. Given a discrete group $G$ and a $G$-space $X$, let $\Coh^G(X)$ be the abelian category of coherent $G$-sheaves on $X$. These are pairs $(\GG, \theta)$, where $\GG$ is a coherent sheaf on $X$ and $\theta$ is a lift of the $G$-action.

Suppose the $G$ action on $X$ is free and let $Y$ be the quotient space, $\pi:X\to Y=X/G$. We have a categorical equivalence between coherent $G$-sheaves on $X$ and coherent sheaves on $Y$, 
\[\Coh^G(X) \to \Coh(Y = X/G).\]
On one hand, given any coherent sheaf $\FF$ on $Y$, its pullback $\pi^*\FF$ is naturally a $G$-sheaf on $X$ and we have a functor $\Coh(Y) \to \Coh^G(X)$. We also have a functor in the other direction. Let $\GG\in \Coh^G(X)$. Then $\pi_*\GG$ has a natural action of $G$ induced from the action of $G$ on $\Coh^G(X)$ and the sheaf of $G$-invariants, $(\pi_*\GG)^G$ is a coherent sheaf on $Y$. This defines an inverse functor $\Coh^G(X) \to \Coh(Y)$.

We now determine $\Pic^0(\Fsing)$, which acts on $\MsingT$ by tensoring.

\begin{definition}Define $P$ to be the degree 0 line bundles on $\Fsing$:
\[
P\coloneqq\Pic^0(\Fsing).
\]
\end{definition}
 
\begin{proposition}\label{prop:fsing bundles} Let $P=\Pic^0(\Fsing)$ as above. Then
\[
P \cong\CC^*\times\CC^*.
\]
\end{proposition}
\begin{proof} Let $\pr:\UFs\to \Fsing$ be the universal cover of $\Fsing$ and let $G=\pi_1 (\Fsing)$ be the fundamental group of $\Fsing$, which acts on $\UFs$ by deck transformations. Recall (Sec.~\ref{geometry}) that the universal cover $\UFs$ is a non-finite type, non-normal toric variety with a countable number of irreducible components, each of which is isomorphic to the surface $\Bl_{a,b}(\PP^1\times\PP^1)$, where $a$ and $b$ are two points in $\PP^1\times\PP^1$. We will call this surface $S\coloneqq\Bl_{a,b}(\PP^1\times\PP^1)$ for this proof. Informally, the universal cover is an infinite union of toric surfaces which are moved around by the deck transformations. 

In our case, $G=\ZZ\times\ZZ$ is generated by two elements $G=\langle e_1, e_2 \rangle$, so a lift of the $G$-action is determined by two commuting isomorphisms, $\mu_i :\GG \to e_i^*\GG$, $i=1,2$. 

A degree zero line bundle $L$ on $\Fsing$ corresponds to the $G$ line bundle $\widetilde{L} = (\pr^* L,\mu_1,\mu_2)$. Here, $\pr^* L$ is a degree zero line bundle. Because $\HH^1(S,\OO_{S})=0$, the line bundles on $S$ are determined by their degree. Since $\pr^* L$ restricted to each irreducible surface component of $\UFs$ is degree zero, it is trivial on each component. Then on ${\UFs}$, which is connected and simply connected, $\pr^* L$ is also trivial. If we choose a trivialization $\pr^* L\cong \UFs \times \CC$, then $\mu_i(x,v)=(e_i(x), \mu_i(x) v)$. Since $\mu_i(x)$ is constant on each irreducible surface component of $\UFs$, $\mu_i(x)=\mu_i\in\CC^*$.

Hence, $\widetilde{L}$ is the triple $(\OO_{\UFs},\mu_1,\mu_2)$, where $\mu_i$ is the map which acts on the fiber by multiplication by a constant, $\mu_i\in\CC^*$. 

These $(\OO_{\UFs},\mu_1,\mu_2)$, for $(\mu_1,\mu_2)\in \CC^*\times\CC^*$ are bijective with isomorphism classes of degree zero line bundles on $\Fsing$, and we get $\CC^*\times\CC^* =denot \Pic^0(\Fsing).$
These $(\OO_{\UFs},\mu_1,\mu_2)$, for $(\mu_1,\mu_2)\in \CC^*\times\CC^*$ are bijective with isomorphism classes of degree zero line bundles on $\Fsing$, and we get $\CC^*\times\CC^* = \Pic^0(\Fsing).$

\end{proof}

We will denote the line bundles on $\Fsing$ constructed in the proof of Proposition~\ref{prop:fsing bundles} as $L_{\bm{\mu}}$:
\[
L_{\bm{\mu}}=L_{\mu_1, \, \mu_2}, \qquad\text{ for }\bm{\mu}=(\mu_1,\mu_2)\in P= \CC^*\times\CC^*. 
\]
We will prove in Proposition \ref{prop:atomic} that the fixed points of the action of tensoring by these degree zero line bundles will correspond in our moduli space to a fixed sheaf. This sheaf is a pushforward of a sheaf on the universal cover.

\begin{definition} Consider the action of $P$ on $\Coh(\Fsing)$ given by $(L_{\mu_1, \, \mu_2},\FF)\mapsto L_{\mu_1, \, \mu_2} \otimes \FF$. Define $\MsingTP$ as those sheaves in the moduli space $\MsingT$ which are also invariant under this action of the torus $P$, 
\[
\MsingTP \coloneqq \{ [\FF]\in \MsingT| L_{\mu_1, \, \mu_2} \otimes \FF \cong \FF \text{ for all } (\mu_1,\mu_2)\in P.\}
\]
\end{definition}

We will also define a moduli space of sheaves on $\UFs$. There is an action on $\UFs$ induced naturally from the action of $T$ on $\Fsing$ and we use the same notation for both.
\begin{definition}\label{defn:MU}
Define $\MU$ to be the moduli space
\[
\MU \coloneqq \left\{ \FFF \in \Coh(\UFs) \Bigg|
\begin{aligned}
&\FFF \text{ pure, one dimension, $T$-invariant,}\\
& \text{stable, }[\Supp(\pr_*\,\FF)]=\beta,\, \chi(\FFF)=1
\end{aligned}
\right\} \big/ \text{ iso}.
\]
\end{definition}

To establish the correspondence between sheaves in $\MsingTP$ supported on $\Fsing$ and those in $\MU$ on $\UFs$, we begin with the following observation.

\begin{remark}
Since $\pr:\UFs \to \Fsing$ is a covering map, the purity and dimension of sheaf support is unchanged under pushforward and pullback. The torus $T$ action on $\UFs$ is by definition the pullback of the torus action on $\Fsing$, so the notion of invariance under the torus action is also preserved. Also, note that the Euler characteristic is preserved under pushforward by this covering map, $\chi(\FFF) = \chi(\pr_*(\FFF))$ for any sheaf $\FFF$ on $\UFs$. 
\label{rmk:cover_preserve_props}
\end{remark}

\begin{proposition}\label{prop:atomic}
Let $\EE \in \MsingTP$. Then there is a $\FFF \in \MU$, unique up to translation by deck transformations, such that $\pr_*\FFF \cong \EE$. 
  
\end{proposition}
\begin{proof} Let $G= \ZZ\times \ZZ=\langle e_1, e_2 \rangle$ act on $\UFs$ by the deck transformations.

Now suppose $\EE\in \MsingTP$. Let $\EEE = \pr^*\EE$. Then $\EEE$ is a $G$-sheaf so it defines a triple, $\{\EEE, \phi_1, \phi_2 \}$, where $\phi_i:\EEE\to e_i^*\EEE$ covers the action of $e_i$ on $\UFs$, so that $[\phi_1,\phi_2]=0$.

The line bundles $L_{\mu_1, \, \mu_2}$ pull back to
\[
\pr^* L_{\mu_1, \, \mu_2}  \cong \{\OO_{\UFs}, \mu_1, \mu_2\},
\]
where each $\mu_i$ is the multiplication by scalar map.

The lift of $\EE \otimes L_{\mu_1, \, \mu_2}$ is then the triple 
\[pr^*(\EE \otimes L_{\mu_1, \, \mu_2}) \coloneqq \{\EEE\otimes\OO_{\UFs},\mu_1\phi_1, \mu_2\phi_2 \} = \{\EEE,\mu_1\phi_1, \mu_2\phi_2 \}.\]

By assumption, $\EE$ satisfies $\EE \otimes L_{\mu_1, \, \mu_2} \cong \EE$, for all $\bm{\mu}=(\mu_1, \mu_2) \in \CC^*\times\CC^*$. This means that we have an isomorphism of $G$-sheaves 
\[
\Psi_{\bm{\mu}}:\{\EEE, \phi_1, \phi_2 \} \cong \{\EEE,\mu_1\phi_1, \mu_2\phi_2 \},
\] 
which induces an automorphism
\[
\psi_{\bm{\mu}}:\EEE \to \EEE,
\]
for all $\bm{\mu}=(\mu_1, \mu_2) \in \CC^*\times\CC^*$.

Combining these, we get a commutative diagram
\[ 
\begin{tikzcd}
\EEE \arrow[r, swap, "\phi_{\mathbf{i}}"'] \arrow[d, swap, "\psi_{\bm{\mu}}"]
& e_{\mathbf{i}}^*\EEE \arrow[d, "e_{\mathbf{i}}^*\psi_{\bm{\mu}}"] \\ 
\EEE \arrow[r, "\bm{\mu}\phi_{\mathbf{i}}"]& e_{\mathbf{i}}^*\EEE
\end{tikzcd}
\]

Since $(\EE\otimes L_{\bm{\mu}}) \otimes L_{\bm{\lambda}} \cong \EE\otimes ( L_{\bm{\mu}} \otimes L_{\bm{\lambda}}) \cong \EE $, both correspond to $\{\EEE,{\bm{\lambda}}{\bm{\mu}}\phi_{\mathbf{i}} \}$. In other words,
\[
\Psi_{\bm{\mu}}\circ\Psi_{\bm{\lambda}} = \Psi_{\bm{\mu}\bm{\lambda}}, \quad \text{and}  \quad
\psi_{\bm{\mu}}\circ\psi_{\bm{\lambda}} = \psi_{\bm{\mu}\bm{\lambda}}
\]
and this defines an action of $\CC^*\times\CC^*$ on $\EEE$,
\[ 
\begin{tikzcd}
\EEE \arrow[r, swap, "\phi_{\mathbf{i}}"'] \arrow[d, swap, "\psi_{\bm{\mu}}"]
& e_{\mathbf{i}}^*\EEE \arrow[d, "e_{\mathbf{i}}^*\psi_{\bm{\mu}}"] \\ 
\EEE \arrow[r, swap, "{\bm{\mu}}\phi_{\mathbf{i}}"'] \arrow[d, swap, "\psi_{\bm{\lambda}}"]
& e_{\mathbf{i}}^*\EEE \arrow[d, "e_{\mathbf{i}}^*\psi_{\bm{\lambda}}"] \\ 
\EEE \arrow[r, "{\bm{\lambda}}{\bm{\mu}}\phi_{\mathbf{i}}"]& e_{\mathbf{i}}^*\EEE
\end{tikzcd}
\]
We define the subsheaves $\EEE_{\mathbf{k}}$ of $\EEE$, as follows:
\begin{gather*}
\EEE_{\mathbf{k}}\coloneqq \Ker(\psi_{\mathbf{\mu}}-\mu_1^{k_1}\mu_2^{k_2}\,{\operatorname{Id}}), \\
{\mathbf{k}}\in\ZZ\times\ZZ,\, \bm{\mu}=(\mu_1,\mu_2)\in \CC^*\times\CC^*.
\end{gather*}
This is independent of choice of $\bm{\mu}$, from the definition given above of the action $\psi_{\mathbf{\mu}}$. 

We will call these subsheaves eigensheaves. We can decompose $\EEE$ into eigensheaves of this torus action,
\[
\EEE=\bigoplus_{{\mathbf{k}}\in\ZZ\times\ZZ}\, \EEE_{\mathbf{k}}.
\]

Restricting to such an eigensheaf then gives the commuting diagram
\[ 
\begin{tikzcd}
\EEE_{\mathbf{k}} \arrow[r, swap, "\phi_{\mathbf{i}}"'] \arrow[d, swap, "\psi_{\bm{\mu}}
"]
& e_{\mathbf{i}}^*\EEE_{\mathbf{k}} \arrow[d, "e_{\mathbf{i}}^*\psi_{\bm{\mu}}"] \\ 
\EEE_{\mathbf{k}} \arrow[r, "\bm{\mu}\phi_{\mathbf{i}}"]& e_{\mathbf{i}}^*\EEE_{\mathbf{k}}
\end{tikzcd}
\]

From this we see there are isomorphisms
\[
(e_1^*)^i(e_2^*)^j \EEE_{k_1, \, k_2}\cong  \EEE_{k_1-i, \, k_2-j},
\]
and the eigensheaves are isomorphic to each other under the action of the deck transformations.

Consider one of these eigensheaves, say $\EEE_{\mathbf{k}}$. From the construction, we see that its pushforward is isomorphic to the original $P$-invariant sheaf  $\pr_*\EEE_{\mathbf{k}}\cong \EE$ on $\Fsing$ and $\pr^*\pr_*\EEE_{\mathbf{k}} =\EEE$. Conversely, if $\FFF$ is such that $\pr_*\FFF=\EE$, then by applying the eigenspace decomposition to $\FFF$, $\FFF$ must be a single summand by stability. This implies uniqueness up to translation.

We now establish that for our sheaves of interest, stability is preserved when moving between $\UFs$ and $\Fsing$. We can use the Euler characteristic characterization of stability, Lemma \ref{lemma: nosemistable-euler}. 

Assume $\FF\in \MsingTP$. For any $\FFF\in\Coh(\UFs)$ such that $\pr_*\FFF = \FF$, let $\EEE\subset\FFF$ be a subsheaf. Then its pushforward is a subsheaf of $\FF$,  $\pr_*\EEE \subset\pr_*\FFF =\FF$. By stability of $\FF$, we have $0\geq\chi(\pr_*(\EEE))\Rightarrow 0\geq\chi(\EEE)$. Thus $\FFF$ is also stable.

For the converse, suppose $\FFF\in \MU$ with $\FF=\pr_*\FFF$. From the construction in Proposition \ref{prop:atomic}, $\FF$ is fixed under the action of $P$. Let $\EE\subset\FF$ be a proper subsheaf. Define $\EE_{\bm{\mu}}\coloneqq \EE\otimes L_{\bm{\mu}}$. Since $\deg (L_{\bm{\mu}}) = 0$, the Euler characteristic is independent of $\bm{\mu}$, $\chi(\EE_{\bm{\mu}})=\chi(\EE)$.

This gives a flat family of coherent sheaves over $(\CC^*\times \CC^*)\times \Fsing$, whose restriction to $\bm{\mu} \times \Fsing$ is $\EE_{\bm{\mu}}$. Our $\Fsing$ is proper, so coherent sheaves satisfy the existence part of the valuative criterion. Thus, we have some limiting sheaf $\EE_{\vec{0}}$, which is invariant under the action of $P$. Then by Proposition \ref{prop:atomic}, there is some $\widetilde{\EE_0}\subset\FFF$ such that $\pr_*\widetilde{\EE_0}=\EE_{\vec{0}}$. Since $\FFF$ is assumed stable, $0\geq\chi(\widetilde{\EE_0}) =\chi(\pr_*\widetilde{\EE_0})\Rightarrow 0\geq \chi(\EE_{\vec{0}}) = \chi(\EE)$, so $\FF$ is also stable.
\end{proof}

In Section~\ref{sec:behrend function}, we will show that this action of $P$ also preserves the symmetric obstruction theory of $M$, and calculate the parity of the tangent space dimensions at the fixed points of this action.

%%%%%%%%%%%%%%%%%%%%%%%%%%%%%%%
%\input{./section4}
\section{Counting Sheaves on $\UFs$}\label{sec: count}

The main result we want to show in this section is the following:
\begin{proposition} Suppose $\FF\in \MU$. Then $\FF\cong\OO_{\CCC}$ for some $T$-invariant curve $\CCC$ with $\chi(\OO_{\CCC})= 1$.
\label{prop:anyrankdeg}
\end{proposition}
In order to prove Proposition~\ref{prop:anyrankdeg}, we need the following key lemma and its corollary, whose proofs we postpone until later.
\begin{lemma}
Let $\FF\in\MU$ with support curve $\Supp\FF=\CCC$. Then $\chi(\OO_{\CCC})\geq 1.$
\label{lemma:structure-chi}
\end{lemma}
\begin{proof}
See Subsection \ref{subsection:lemmaproof}.
\end{proof}
As a Corollary, we have 
\begin{corollary}
Let $\FF\in\MU$ with support curve $\Supp\FF=\CCC$. Let $\DDD$ be any closed subscheme of $\CCC$. Then $\chi(\OO_{\DDD})\geq 1$.
\label{cor:subscheme-chi}
\end{corollary}
\begin{proof}
See Subsection \ref{subsection:corproof}.
\end{proof}
Using Corollary~\ref{cor:subscheme-chi}, we can prove Proposition~\ref{prop:anyrankdeg}.
\begin{proof}[Proof (of Proposition~\ref{prop:anyrankdeg}).]
Let $\FF\in \MU$ and $\CCC=\Supp\FF$. By hypothesis, $\chi(\FF)=1$ so $\FF$ has a nonzero global section $s$. Let $\II$ be the kernel of the map $s$. Then we have the exact sequence:
\[
0\to\II\to\OO_C\stackrel{s}{\rightarrow}\FF
\]
Let $\CCC_s\subset \CCC$ be the support of $s$. Then $\OO_{\CCC_s} = \OO_{\CCC} /\II$ is a subsheaf of $\FF$,
\[
0\to\OO_{\CCC_s}\to\FF
\]
By Corollary~\ref{cor:subscheme-chi}, $\chi(\OO_{\CCC_s}) \geq1$, which contradicts Lemma~\ref{lemma: nosemistable-euler} unless $\OO_{\CCC_s} \cong \OO_{\CCC} \cong \FF$.
\end{proof}
\subsection{Formula for Euler characteristic.}\label{subsection:master}
Before we present the proof of Lemma~\ref{lemma:structure-chi} and Corollary~\ref{cor:subscheme-chi}, we derive a formula to compute the Euler characteristic of structure sheaves of a certain type of curve in $\UFs$.

Since we are interested in sheaves on $\Xb$ with support in class $\beta=d_1[C_1]+d_2[C_2]+[C_3]$, we can eliminate any isomorphisms induced by the deck transformations on $\UFs$ by fixing a curve that lies over $C_3\subset\Cban$. By Proposition~\ref{prop:atomic}, any point in $\MU$ can be uniquely represented by a sheaf whose support contains this curve. 

In the discusion of Section~\ref{localbanana}, we observed that the formal neighborhood of any irreducible component $e$ that covers $C_i$ is formally locally isomorphic to the total space of $\OO(-1) \oplus \OO(-1) \to \PP^1$. As a consequence, we have a map from $e$ to the reduced curve $e_{\textsc{red}}$ in our geometry.

We record these observations in the terminology that we will use in this section.

\begin{notation}
Suppose $\FF\in\MU$.
\begin{itemize}
\item $e_0\cong\PP^1\subset\UFs$ is a fixed curve such that $\pr_*(e_0)=[C_3]$.
\item $\Gamma\coloneqq\{\CCC \subset \UFs\, |\, \dim\CCC=1,\text{ connected, $T$-fixed, } e_0\subset\CCC, [\pr_*(\CCC)]=d_1[C_1]+d_2[C_2]+[C_3], d_1,d_2\geq 0 \}$, see Figure \ref{fig:gamma}.
\item $\CCC\coloneqq\Supp\FF$ so $[(\pr_*\,\CCC)]=\beta$. Without loss of generality, we will let $\FF$ be such that $\CCC\in\Gamma$.
\item Edges $\{e_i\}$ are possibly nonreduced, irreducible components of $\CCC$.
\item Vertices $\{v_j\}$ are points where two or more components of $\CCC$ intersect. 
\item $\phi:e\to e_{\textsc{red}}\cong\PP^1$ is the map that exists for edges $e$ in our geometry.
\end{itemize}
\end{notation}

\begin{figure}[th!] 
  \centering
	\scalebox{1}{
	\begin{tikzpicture}[y=0.80pt, x=0.80pt, yscale=-1.000000, xscale=1.000000, inner sep=0pt, outer sep=0pt]
\begin{scope}[shift={(0,0)}]% layer1
  % flowRoot4196
\node at (43,90) {$\scriptstyle e_0$};
%\node at (20,95) {\scalebox{.7}{$\scriptstyle p_0$}};
%\node at (65,95) {\scalebox{.7}{$\scriptstyle q_0$}};
%\node at (76,70) {\scalebox{.7}{$\scriptstyle p$}};
%\node at (62,45) {\scalebox{.7}{$\scriptstyle q$}};

%\node at (60,80) {\scalebox{.7}{\rotatebox{60}{$\scriptstyle A$}}};
%\node at (62,60) {\scalebox{.7}{\rotatebox{120}{$\scriptstyle B$}}};

%\node at (60,30) {\scalebox{.7}{\rotatebox{60}{$\scriptstyle A$}}};	
	
  \path[fill=black,line join=miter,line cap=butt,line width=0.800pt]
    (0.0000,0.0000) node[above right] (flowRoot4196) {};

  \begin{scope}[shift={(-42.88279,-53.05836)}]% g5827
    % path4268
    \path[draw=black,line join=miter,line cap=butt,even odd rule,line width=0.800pt]
      (115.1939,72.9186) -- (100.1939,97.9186) -- (115.1939,122.9186) --
      (100.1939,147.9186) -- (115.1939,172.9186) -- (100.1939,197.9186) --
      (115.1939,222.9186);

    % path4281
    \path[draw=black,line join=miter,line cap=butt,even odd rule,line width=0.800pt]
      (55.5456,222.9186) -- (70.5456,197.9186) -- (55.5456,172.9186) --
      (70.5456,147.9186) -- (55.5456,122.9186) -- (70.5456,97.9186) --
      (55.5456,72.9186)(100.1939,147.9186) -- (70.1939,147.9186);

    \begin{scope}[shift={(-4.24264,-37.47666)}]% g4202
      % use4160
      \path[draw=black,fill=black,line join=round,line cap=rect,miter limit=4.00,line
        width=0.200pt] (62.5000,267.6122) circle (0.0176cm);

      % use4164
      \path[draw=black,fill=black,line join=round,line cap=rect,miter limit=4.00,line
        width=0.200pt] (62.5000,277.6122) circle (0.0176cm);

      % use4176
      \path[draw=black,fill=black,line join=round,line cap=rect,miter limit=4.00,line
        width=0.200pt] (62.5000,272.6122) circle (0.0176cm);

    \end{scope}
    \begin{scope}[shift={(48.33274,-38.18376)}]% g4202-4
      % use4160-8
      \path[draw=black,fill=black,line join=round,line cap=rect,miter limit=4.00,line
        width=0.200pt] (62.5000,267.6122) circle (0.0176cm);

      % use4164-8
      \path[draw=black,fill=black,line join=round,line cap=rect,miter limit=4.00,line
        width=0.200pt] (62.5000,277.6122) circle (0.0176cm);

      % use4176-8
      \path[draw=black,fill=black,line join=round,line cap=rect,miter limit=4.00,line
        width=0.200pt] (62.5000,272.6122) circle (0.0176cm);

    \end{scope}
    \begin{scope}[shift={(47.62563,-210.15253)}]% g4202-3
      % use4160-5
      \path[draw=black,fill=black,line join=round,line cap=rect,miter limit=4.00,line
        width=0.200pt] (62.5000,267.6122) circle (0.0176cm);

      % use4164-9
      \path[draw=black,fill=black,line join=round,line cap=rect,miter limit=4.00,line
        width=0.200pt] (62.5000,277.6122) circle (0.0176cm);

      % use4176-3
      \path[draw=black,fill=black,line join=round,line cap=rect,miter limit=4.00,line
        width=0.200pt] (62.5000,272.6122) circle (0.0176cm);

    \end{scope}
    \begin{scope}[shift={(-1.61522,-210.85964)}]% g4202-1
      % use4160-4
      \path[draw=black,fill=black,line join=round,line cap=rect,miter limit=4.00,line
        width=0.200pt] (62.5000,267.6122) circle (0.0176cm);

      % use4164-80
      \path[draw=black,fill=black,line join=round,line cap=rect,miter limit=4.00,line
        width=0.200pt] (62.5000,277.6122) circle (0.0176cm);

      % use4176-4
      \path[draw=black,fill=black,line join=round,line cap=rect,miter limit=4.00,line
        width=0.200pt] (62.5000,272.6122) circle (0.0176cm);

    \end{scope}
    % use4160-43
    \path[draw=black,fill=black,line join=round,line cap=rect,miter limit=4.00,line
      width=0.389pt] (100.3231,147.5099) circle (0.0343cm);

    % use4160-43-7
    \path[draw=black,fill=black,line join=round,line cap=rect,miter limit=4.00,line
      width=0.389pt] (100.6502,197.8065) circle (0.0343cm);

    % use4160-43-7-4
    \path[draw=black,fill=black,line join=round,line cap=rect,miter limit=4.00,line
      width=0.389pt] (114.7685,123.2673) circle (0.0343cm);

    % use4160-43-7-9
    \path[draw=black,fill=black,line join=round,line cap=rect,miter limit=4.00,line
      width=0.389pt] (100.5252,97.4315) circle (0.0343cm);

    % use4160-43-7-3
    \path[draw=black,fill=black,line join=round,line cap=rect,miter limit=4.00,line
      width=0.389pt] (70.0252,98.0565) circle (0.0343cm);

    % use4160-43-7-40
    \path[draw=black,fill=black,line join=round,line cap=rect,miter limit=4.00,line
      width=0.389pt] (55.6502,123.1815) circle (0.0343cm);

    % use4160-43-7-0
    \path[draw=black,fill=black,line join=round,line cap=rect,miter limit=4.00,line
      width=0.389pt] (70.3751,147.8132) circle (0.0343cm);

    % use4160-43-7-8
    \path[draw=black,fill=black,line join=round,line cap=rect,miter limit=4.00,line
      width=0.389pt] (115.2486,172.8635) circle (0.0343cm);

    % use4160-43-7-91
    \path[draw=black,fill=black,line join=round,line cap=rect,miter limit=4.00,line
      width=0.389pt] (70.4002,198.1815) circle (0.0343cm);

    % use4160-43-7-65
    \path[draw=black,fill=black,line join=round,line cap=rect,miter limit=4.00,line
      width=0.389pt] (55.5016,172.6867) circle (0.0343cm);

  \end{scope}
\end{scope}

\end{tikzpicture}
	
	}	
  \caption{Example of a curve $\CCC\subset \UFs$, $\CCC\in\Gamma$}
	\label{fig:gamma}
\end{figure}

We can write the support of $\FF$ as
\[ \CCC\coloneqq\Supp\FF=\bigcup_{ e_i\in \{\text{edges of } \CCC \} } e_i,\] 
where each $e_i$ is a component with a unique torus invariant thickening on $C$ determined by two numbers on each edge, $m_e, n_e$ (see Section \ref{localbanana}), and $(e_i)_{\text{red}} \cong \PP^1$.

The following is a special case of \cite[Lemma~5]{mnop1}. We give a self-contained proof for convenience, and tailor it to our situation and notation to obtain a formula to compute Euler characteristics of structure sheaves of support curves of $\FF\in\MU$.

\begin{lemma}\cite{mnop1}\label{lemma:masterformula}
Let $\CCC$ be a connected pure one dimensional curve in $\Gamma$. Let $\{e_i\}$ be the edges and $\{v_j\}$ the vertices of $\CCC$.

Then $\chi(\OO_{\CCC})$ satisfies 
\begin{equation}
\chi(\OO_{\CCC})= \sum_{e_i}{E(\OO_{\CCC},e_i)}  - \sum_{v_j}{V(\OO_{\CCC},v_j)},
\label{eq:inclusionexclusion}
\end{equation}
where $E(\OO_{\CCC},e_i)$ and $V(\OO_{\CCC},v_j)$ are integer valued functions on the edges and vertices, respectively, and defined as follows:

Given an edge $e$ with thickening lengths $m$ and $n$, the integer $E(\OO_{\CCC},e)$ is given by
\[
E(\OO_{\CCC},e) ={\Bigg[\binom{m+1}{2} + \binom{n+1}{2} - 1\Bigg]}.
\] 

At a vertex $v$ with three incident edges and multiple structure as in Figure~\ref{fig:pile of boxes}, then the integer $V(\OO_{\CCC},v)$ is given by
\[
V(\OO_{\CCC},v) = (mr + sa + bn  - 1).
\] 

If $v$ only has two incident edges, corresponding to, say, the $x$ and $y$ axes with thickenings as in Figure~\ref{fig:pile of boxes}, then
\[
V(\OO_{\CCC},v)= (mr + \min(n,s) -1).
\]
\end{lemma}

\begin{proof}
This is a computation of Euler characteristic using the normalization sequence, and by pushing forward sheaves on irreducible components to their reduced counterparts.

Consider the normalization sequence, 
\[
0 \to \OO_{\CCC} \to \bigoplus_{i} \OO_{\CCC}|_{e_i} \to \bigoplus_{i\neq j}\OO_{\CCC}|_{e_i \cap e_j} \to \bigoplus_{\substack{i,j,k\\ \text{distinct}}}\OO_{\CCC}|_{e_i\cap e_j\cap e_k} \to 0.
\]
Then 
\[
\chi(\OO_{\CCC}) = \bigoplus_{i} \chi(\OO_{\CCC}|_{e_i}) \,-\, \bigoplus_{i\neq j}\chi(\OO_{\CCC}|_{e_i\cap e_j})\, +  \bigoplus_{\substack{i,j,k\\ \text{distinct}}}\chi(\OO_{\CCC}|_{e_i\cap e_j\cap e_k}).
\]

First, we calculate the Euler characteristic of the restriction of our sheaf to a single edge. Let $e \subset \CCC$ be an edge with thickening lengths $m, n$ and map to the reduced curve, $\phi: e \to e_{\textsc{red}}\cong\PP^1$.

Then $\phi$ has zero dimensional fiber, so
\[\chi(\OO_e) = \chi(\phi^*\OO_{\PP^1}) = \chi(\phi_*\OO_e)\]
by the projection formula.

The normal bundle $\NNN$ of $e$ in $\UFs$ is formally locally isomorphic to a variety affine over $e_{\mathsc{red}}\cong\PP^1$:
\[\NNN\coloneqq\Tot(\OO(-1)\oplus\OO(-1)\to\PP^1),\]
and its sheaf of algebras over $\PP^1$ is given by
\[
\phi_*\OO_{\NNN} = \Sym^*\NNN^{\vee} = \bigoplus_{i,j} \OO(i)\otimes\OO(j) .
\]
So we can think of $\phi_*\OO_e$ as a quotient of $\phi_*\OO_{\NNN}$.

We may represent these summands graphically by boxes in the first quadrant labeled by monomial generators. The quotient sheaf $\phi_*\OO_e$ with lengths $m$ and $n$ along the axes then corresponds to the diagram in Figure~\ref{fig:el}.

\begin{figure}[th!]
  \centering
	\scalebox{1}{
	\begin{tikzpicture}[y=0.80pt, x=0.80pt, yscale=-1.000000, xscale=1.000000, inner sep=0pt, outer sep=0pt]
\begin{scope}[shift={(0,0)}]
\node at (19,20) {\scalebox{1}{$y^{\scriptscriptstyle{m-1}}$}};
\node at (19,50) {\scalebox{1}{$\vdots$}};
\node at (18,86) {\scalebox{1}{$y^{\scriptscriptstyle{2}}$}};
\node at (18,121) {\scalebox{1}{$y^{\phantom{1}}$}};
\node at (18,156) {\scalebox{1}{$1^{\phantom{1}}$}};

\node at (55,156) {$x^{\phantom{1}}$};
\node at (86,156) {$\cdots$};

\node at (122,156) {$x^{\scriptscriptstyle{n-1}}$};

  \path[fill=black,line join=miter,line cap=butt,line width=0.800pt]
    (0.0000,0.0000) node[above right] (flowRoot4196) {};
  \begin{scope}[cm={{0.0,-0.84243,-0.84243,0.0,(176.87678,212.96868)}}]
    \begin{scope}[shift={(-228.86343,-235.5947)}]
      \begin{scope}[shift={(-12.98175,82.69855)}]
        \path[draw=black,fill=black,miter limit=4.00,fill opacity=0.039,line
          width=2.282pt,rounded corners=0.0000cm] (329.9220,320.9198) rectangle
          (369.9220,360.9198);
        \begin{scope}[shift={(40.7094,0)},shift={(0,0)}]
          \path[draw=black,fill=black,miter limit=4.00,fill opacity=0.039,line
            width=2.282pt,rounded corners=0.0000cm] (329.9220,320.9198) rectangle
            (369.9220,360.9198);
        \end{scope}
        \begin{scope}[shift={(81.4188,0)},shift={(0,0)}]
          \path[draw=black,fill=black,miter limit=4.00,fill opacity=0.039,line
            width=2.282pt,rounded corners=0.0000cm] (329.9220,320.9198) rectangle
            (369.9220,360.9198);
        \end{scope}
        \begin{scope}[shift={(122.1282,0)},shift={(0,0)}]
          \path[draw=black,fill=black,miter limit=4.00,fill opacity=0.039,line
            width=2.282pt,rounded corners=0.0000cm] (329.9220,320.9198) rectangle
            (369.9220,360.9198);
        \end{scope}
      \end{scope}
      \begin{scope}[shift={(-272.61673,-11.54293)}]
        \begin{scope}[shift={(218.76651,-27.88672)},shift={(0,0)}]
          \path[draw=black,fill=black,miter limit=4.00,fill opacity=0.039,line
            width=2.282pt,rounded corners=0.0000cm] (329.9220,320.9198) rectangle
            (369.9220,360.9198);
        \end{scope}
        \begin{scope}[shift={(0,40.7094)},shift={(0,0)}]
          \begin{scope}[shift={(218.76651,-27.88672)},shift={(0,0)}]
            \path[draw=black,fill=black,miter limit=4.00,fill opacity=0.039,line
              width=2.282pt,rounded corners=0.0000cm] (329.9220,320.9198) rectangle
              (369.9220,360.9198);
          \end{scope}
        \end{scope}
        \begin{scope}[shift={(0,81.4188)},shift={(0,0)}]
          \begin{scope}[shift={(218.76651,-27.88672)},shift={(0,0)}]
            \path[draw=black,fill=black,miter limit=4.00,fill opacity=0.039,line
              width=2.282pt,rounded corners=0.0000cm] (329.9220,320.9198) rectangle
              (369.9220,360.9198);
          \end{scope}
        \end{scope}
        \begin{scope}[shift={(0,122.1282)},shift={(0,0)}]
          \begin{scope}[shift={(218.76651,-27.88672)},shift={(0,0)}]
            \path[draw=black,fill=black,miter limit=4.00,fill opacity=0.039,line
              width=2.282pt,rounded corners=0.0000cm] (329.9220,320.9198) rectangle
              (369.9220,360.9198);
          \end{scope}
        \end{scope}
      \end{scope}
    \end{scope}
  \end{scope}
\end{scope}

\end{tikzpicture}

	}
  \caption{$\phi_*\OO_e$ breaks up into a sum of line bundles on $\PP^1$.}
	\label{fig:el}
\end{figure}

In other words, 
\[
\phi_*\OO_e = \OO\oplus\bigoplus_{i=1}^{m-1} \OO(i) \oplus \bigoplus_{j=1}^{n-1} \OO(j).
\]
A straightforward application of Riemann-Roch gives, 
\[\chi(\OO_e)=\chi(\phi_*\OO_e) = \left[\binom{m+1}{2} + \binom{n+1}{2} -1 \right].
\]
Then $\bigoplus_{e_i} \chi(\OO_{\CCC}|_{e_i}) = \sum_{e_i}{E(\OO_{\CCC},e_i)}$, where the edge contribution $E(\OO_{\CCC},e)$ to Equation~\ref{eq:inclusionexclusion} is of the required form.

In order to calculate the Euler characteristic contribution from the vertex terms, we need to express the lengths of the module at a vertex in terms of the thickenings of the incident edges. If we depict the multiple structure at a vertex in terms of boxes representing monomial ideals, we must count the number of common boxes shared by pairwise edges and subtract the contribution from boxes common to all three incident edges. 

For example, suppose the vertex $v$ has three incident edges $e_x, e_y, e_z$, with multiple structure labeled as in Figure \ref{fig:pile of boxes}. Then for a given pair of edges, say $e_x, e_y$, the number of boxes they share is
\[
\chi(\OO_{e_x\cap e_y}) = mr - \min(n,s) -1.
\]
For a vertex with three incident edges, the pairwise intersections contribute the following total to the Euler characteristic:
\begin{align*}
&\chi(\OO_{e_x\cap e_y}) + \chi(\OO_{e_y\cap e_z}) + \chi(\OO_{e_z\cap e_x})\\
&= (mr + \min(n,s) -1) + (sa + \min(r,b) -1) + (bn + \min(a,m) -1).
\end{align*}
On the other hand, the boxes that are in the triple intersection have length:
\[
\chi(\OO_{e_x\cap e_y \cap e_z}) = \min(n,s) + \min(r,b) + \min(a,m) - 2.
\]
Subtracting these expressions gives the contribution of a vertex with three edges. 
\[
\chi(\OO_{e_x\cap e_y}) + \chi(\OO_{e_y\cap e_z}) + \chi(\OO_{e_z\cap e_x}) - \chi(\OO_{e_x\cap e_y \cap e_z}) = rm + ns + ab -1.
\]
Hence,
\[
\bigoplus_{i\neq j}\chi(\OO_{\CCC}|_{e_i\cap e_j})\, -  \bigoplus_{\substack{i,j,k\\ \text{distinct}}}\chi(\OO_{\CCC}|_{e_i\cap e_j\cap e_k}) = \sum_{v_j}{V(\OO_{\CCC},v_j)}
\]
if we define the function $V (\OO_{\CCC}, v) = (rm + ns + ab -1)$ when $v$ has three incident edges, or $V (\OO_{\CCC}, v) = (mr - \min(n,s) -1)$ if it has two, as claimed.
\end{proof}

\subsection{Proof of Lemma~\ref{lemma:structure-chi}.}\label{subsection:lemmaproof}

Let $\CCC$ be a connected curve in $\Gamma$ containing $e_0$. Then $\CCC$ naturally breaks up into a union of four branches characterized by their attachment type to $e_0$. Since the Euler characteristic computation on each branch is identical, we will use this decomposition of the curve to simplify the presentation of the proof of Lemma~\ref{lemma:structure-chi}. To this end, we establish the following nomenclature conventions.

\subsubsection{Terminology}
\textit{Branches:} The space $\overline{\CCC\setminus e_0}$ consists of edges that lie over $C_1$ or $C_2 \subset \Cban$. We can write $\overline{\CCC\setminus e_0}$ as a disjoint union of (possibly empty) connected subcurves of four types: 
\[
\overline{\CCC\setminus e_0} = \CCC^{\sbullet}\amalg \CCC_{\sbullet}\amalg {\vphantom{\CCC}}^{\sbullet}\CCC\amalg {\vphantom{\CCC}}_{\sbullet}\CCC
\]
These are distinguished by their attachment to $e_0$. The edge of the subcurve that intersects $e_0$ can cover either $C_1$ or $C_2$, and the intersection vertex can cover $p$ or $q$. For concreteness, we choose the identifications as indicated in Figure~\ref{fig:branchschematic}. We will call any of these four subcurves $\CCC^{\sbullet}, \CCC_{\sbullet}, {\vphantom{\CCC}}^{\sbullet}\CCC$, or ${\vphantom{\CCC}}_{\sbullet}\CCC$, the branches of $\CCC$.

Likewise, the notation $\CCC_{\sbullet}^{\sbullet}$ will mean the subcurve $\CCC_{\sbullet}^{\sbullet}=\CCC^{\sbullet}\cup\CCC_{\sbullet}$ and so forth. Edges and vertices will be decorated as needed to indicate the branch they are on.

\begin{figure}[th!]
  \centering
	\scalebox{1.5}{
	\begin{tikzpicture}[y=0.80pt, x=0.80pt, yscale=-1.000000, xscale=1.000000, inner sep=0pt, outer sep=0pt]

\node at (43,105) {\scalebox{.7}{$\scriptstyle e_0$}};
\node at (15,95) {\scalebox{.7}{$\scriptstyle p_0$}};
\node at (67,95) {\scalebox{.7}{$\scriptstyle q_0$}};

\node at (27,106) {\scalebox{.7}{$\scriptstyle p_0$}};
\node at (60,106) {\scalebox{.7}{$\scriptstyle q_0$}};

\node at (17,123) {\scalebox{.7}{$\scriptstyle p_0$}};
\node at (67,123) {\scalebox{.7}{$\scriptstyle q_0$}};

\node at (82,70) {\scalebox{.7}{$\scriptstyle p_1$}};
\node at (67,44) {\scalebox{.7}{$\scriptstyle q_1$}};

%\node at (66,78) {\scalebox{.7}{\rotatebox{60}{$\scriptstyle e_{1}^{\sbullet}$}}};
%\node at (62,58) {\scalebox{.7}{\rotatebox{-60}{$\scriptstyle e_{2}^{\sbullet}$}}};
%\node at (64,30) {\scalebox{.7}{\rotatebox{60}{$\scriptstyle e_{3}^{\sbullet}$}}};
\node at (66,78) {\scalebox{.7}{$\scriptstyle e_{1}^{\sbullet}$}};
\node at (64,58) {\scalebox{.7}{$\scriptstyle e_{2}^{\sbullet}$}};
\node at (64,30) {\scalebox{.7}{$\scriptstyle e_{3}^{\sbullet}$}};
\node at (100,55) {$\CCC^{\sbullet}$};

\node at (100,165) {$\CCC_{\sbullet}$};

%\node at (20,137) {\scalebox{.7}{\rotatebox{60}{$\scriptstyle {}_{'}\!A_1$}}};
\node at (-20,165) {${\vphantom{\CCC}}_{\sbullet}\CCC$};
%\node at (15,78) {\scalebox{.7}{\rotatebox{-60}{$\scriptstyle '\! B_1$}}};	
\node at (-20,55) {${\vphantom{\CCC}}^{\sbullet}\CCC$};

\node at (41,50) {\scalebox{.5}{\rotatebox{90}{inside}}};
\node at (80,50) {\scalebox{.5}{\rotatebox{90}{outside}}};
\node at (0,50) {\scalebox{.5}{\rotatebox{90}{outside}}};
\node at (41,170) {\scalebox{.5}{\rotatebox{90}{inside}}};
\node at (80,170) {\scalebox{.5}{\rotatebox{90}{outside}}};
\node at (0,170) {\scalebox{.5}{\rotatebox{90}{outside}}};

\begin{scope}[shift={(0,0)}]% layer1
  % flowRoot4196
  \path[fill=black,line join=miter,line cap=butt,line width=0.800pt]
    (0.0000,0.0000) node[above right] (flowRoot4196) {};

  \begin{scope}% g4298
    % path4281-8
    \path[draw=black,dash pattern=on 0.80pt off 0.80pt,line join=miter,line
      cap=butt,miter limit=4.00,even odd rule,line width=0.800pt] (20.3676,95.2820)
      -- (5.3676,70.2820) -- (20.3676,45.2820) -- (5.3676,20.2820);

    \begin{scope}[shift={(-51.79323,-263.49619)}]% g4202-1-6
      % use4160-4-0
      \path[draw=black,fill=black,line join=round,line cap=rect,miter limit=4.00,line
        width=0.200pt] (62.5000,267.6122) circle (0.0176cm);

      % use4164-80-0
      \path[draw=black,fill=black,line join=round,line cap=rect,miter limit=4.00,line
        width=0.200pt] (62.5000,277.6122) circle (0.0176cm);

      % use4176-4-4
      \path[draw=black,fill=black,line join=round,line cap=rect,miter limit=4.00,line
        width=0.200pt] (62.5000,272.6122) circle (0.0176cm);

    \end{scope}
    % use4160-43-7-3-7
    \path[draw=black,fill=black,line join=round,line cap=rect,miter limit=4.00,line
      width=0.389pt] (19.8472,45.4200) circle (0.0343cm);

    % use4160-43-7-40-5
    \path[draw=black,fill=black,line join=round,line cap=rect,miter limit=4.00,line
      width=0.389pt] (5.4722,70.5450) circle (0.0343cm);

    % use4160-43-7-0-3
    \path[draw=black,fill=black,line join=round,line cap=rect,miter limit=4.00,line
      width=0.389pt] (20.1971,95.1767) circle (0.0343cm);

  \end{scope}
  \begin{scope}[shift={(1.12065,8.12784)}]% g4699
    % use4160-43-6
    \path[draw=black,fill=black,line join=round,line cap=rect,miter limit=4.00,line
      width=0.389pt] (56.1218,102.2158) circle (0.0343cm);

    % path4281-8-1
    \path[draw=black,line join=miter,line cap=butt,even odd rule,line width=0.800pt]
      (55.9926,102.2158) -- (25.9926,102.2158);

    % circle4589
    \path[draw=black,fill=black,line join=round,line cap=rect,miter limit=4.00,line
      width=0.389pt] (25.8003,102.2158) circle (0.0343cm);

  \end{scope}
  \begin{scope}% g4278
    % path4281-8-0
    \path[draw=black,dash pattern=on 1.60pt off 1.60pt,line join=miter,line
      cap=butt,miter limit=4.00,even odd rule,line width=0.800pt] (22.7831,122.6173)
      -- (7.7831,147.6173) -- (22.7831,172.6173) -- (7.7831,197.6173);

    % use4160-43-7-3-7-3
    \path[xscale=1.000,yscale=-1.000,draw=black,fill=black,dash pattern=on 0.78pt
      off 0.78pt,line join=round,line cap=rect,miter limit=4.00,line width=0.389pt]
      (22.2627,-172.4793) circle (0.0343cm);

    % use4160-43-7-40-5-3
    \path[xscale=1.000,yscale=-1.000,draw=black,fill=black,dash pattern=on 0.78pt
      off 0.78pt,line join=round,line cap=rect,miter limit=4.00,line width=0.389pt]
      (7.8877,-147.3543) circle (0.0343cm);

    % use4160-43-7-0-3-3
    \path[xscale=1.000,yscale=-1.000,draw=black,fill=black,dash pattern=on 0.78pt
      off 0.78pt,line join=round,line cap=rect,miter limit=4.00,line width=0.389pt]
      (22.6126,-122.7226) circle (0.0343cm);

    \begin{scope}[shift={(-51.38395,-63.33641)}]% g4202-1-6-5
      % use4160-4-0-4
      \path[draw=black,fill=black,line join=round,line cap=rect,miter limit=4.00,line
        width=0.200pt] (62.5000,267.6122) circle (0.0176cm);

      % use4164-80-0-6
      \path[draw=black,fill=black,line join=round,line cap=rect,miter limit=4.00,line
        width=0.200pt] (62.5000,277.6122) circle (0.0176cm);

      % use4176-4-4-93
      \path[draw=black,fill=black,line join=round,line cap=rect,miter limit=4.00,line
        width=0.200pt] (62.5000,272.6122) circle (0.0176cm);

    \end{scope}
  \end{scope}
  \begin{scope}% g4268
    % path4281-8-4
    \path[draw=black,dash pattern=on 1.60pt off 0.80pt on 0.40pt off 0.80pt,line
      join=miter,line cap=butt,miter limit=4.00,even odd rule,line width=0.800pt]
      (61.9952,122.9050) -- (76.9952,147.9050) -- (61.9952,172.9050) --
      (76.9952,197.9050);

    % use4160-43-7-3-7-1
    \path[scale=-1.000,draw=black,fill=black,dash pattern=on 0.39pt off 1.17pt,line
      join=round,line cap=rect,miter limit=4.00,line width=0.389pt]
      (-62.5157,-172.7671) circle (0.0343cm);

    % use4160-43-7-40-5-2
    \path[scale=-1.000,draw=black,fill=black,dash pattern=on 0.39pt off 1.17pt,line
      join=round,line cap=rect,miter limit=4.00,line width=0.389pt]
      (-76.8907,-147.6421) circle (0.0343cm);

    % use4160-43-7-0-3-7
    \path[scale=-1.000,draw=black,fill=black,dash pattern=on 0.39pt off 1.17pt,line
      join=round,line cap=rect,miter limit=4.00,line width=0.389pt]
      (-62.1658,-123.0104) circle (0.0343cm);

    \begin{scope}[shift={(9.42618,-64.03683)}]% g4202-1-6-8
      % use4160-4-0-48
      \path[draw=black,fill=black,line join=round,line cap=rect,miter limit=4.00,line
        width=0.200pt] (62.5000,267.6122) circle (0.0176cm);

      % use4164-80-0-9
      \path[draw=black,fill=black,line join=round,line cap=rect,miter limit=4.00,line
        width=0.200pt] (62.5000,277.6122) circle (0.0176cm);

      % use4176-4-4-6
      \path[draw=black,fill=black,line join=round,line cap=rect,miter limit=4.00,line
        width=0.200pt] (62.5000,272.6122) circle (0.0176cm);

    \end{scope}
  \end{scope}
  \begin{scope}% g4288
    % path4281-8-7
    \path[draw=black,dash pattern=on 0.20pt off 0.20pt,line join=miter,line
      cap=butt,miter limit=4.00,even odd rule,line width=0.800pt] (61.5510,94.9942)
      -- (76.5510,69.9942) -- (61.5510,44.9942) -- (76.5510,19.9942);

    % use4160-43-7-3-7-34
    \path[xscale=-1.000,yscale=1.000,draw=black,fill=black,line join=round,line
      cap=rect,miter limit=4.00,line width=0.389pt] (-62.0715,45.1322) circle
      (0.0343cm);

    % use4160-43-7-40-5-0
    \path[xscale=-1.000,yscale=1.000,draw=black,fill=black,line join=round,line
      cap=rect,miter limit=4.00,line width=0.389pt] (-76.4465,70.2572) circle
      (0.0343cm);

    % use4160-43-7-0-3-78
    \path[xscale=-1.000,yscale=1.000,draw=black,fill=black,line join=round,line
      cap=rect,miter limit=4.00,line width=0.389pt] (-61.7216,94.8889) circle
      (0.0343cm);

    \begin{scope}[shift={(8.61009,-263.75285)}]% g4202-1-6-3
      % use4160-4-0-5
      \path[draw=black,fill=black,line join=round,line cap=rect,miter limit=4.00,line
        width=0.200pt] (62.5000,267.6122) circle (0.0176cm);

      % use4164-80-0-8
      \path[draw=black,fill=black,line join=round,line cap=rect,miter limit=4.00,line
        width=0.200pt] (62.5000,277.6122) circle (0.0176cm);

      % use4176-4-4-1
      \path[draw=black,fill=black,line join=round,line cap=rect,miter limit=4.00,line
        width=0.200pt] (62.5000,272.6122) circle (0.0176cm);

    \end{scope}
  \end{scope}
\end{scope}

\end{tikzpicture}

	}
  \caption{Schematic diagram of $\CCC$}
	\label{fig:branchschematic}
\end{figure}

\textit{Numbering}: Let $|\CCC|$ denote the number of edges of any curve $\CCC$. We number the consecutive edges of each branch in increasing order away from $e_0$ and group them in \textit{consecutive pairs}, labeled as $(e_{2k-1}, e_{2k})$, $k\geq1$.

\textit{Thickening}: Recall that as a consequence of Proposition~$\ref{prop:tfixcount}$, the sheaves we are interested in have scheme-theoretic support in the surface $\Fsing$. Thus, any thickening of support curves for sheaves in $\MU$ will occur in the surface $\UFs$. The fixed edge $e_0$ is the intersection of two irreducible surface components of $\UFs$. Let $S$ be one of these irreducible surface components containing $e_0$, and let $g$ be the deck transformation which translates $S$ into the other component containing $e_0$. Then $g$ generates a $\ZZ$ subgroup of the deck transformations, $\langle g\rangle\cong\ZZ\subset\ZZ\times\ZZ$. We define the \textit{inside hexagons} as those irreducible surface components of $\UFs$ that are in the orbit of $S$ under the action of $\langle g\rangle$. Any other irreducible surface components will be called \textit{outside hexagons}. 

Every edge of $\CCC$, apart from $e_0$, is the intersection of an inside hexagon and an outside hexagon, and can be thickened in either of these two directions. We will call an edge thickening in the direction of the inside hexagon surface the \textit{inside} direction. The thickening of an edge in the outside hexagon surface direction will be called the \textit{outside} direction. 

\textit{One branch detail:} We choose one branch, say $\CCC^{\sbullet}$, for detailed computations, Figure~\ref{fig:branchdetail}. Here, we will denote the lengths of the multiple structure on edges $e_{2i-1}^{\sbullet}$ by $m_i$ on the inside and $n_i$ on the outside. Edges $e_{2i}^{\sbullet}$ will have multiple structures of lengths $r_i$ on the inside and $s_i$ on the outside. The vertices will be numbered so that $p_i=e_{2i-1} \cap e_{2i}$ and $q_i=e_{2i} \cap e_{2i+1}$. 

\begin{figure}[th!]
  \centering
	\scalebox{1.5}{
	\begin{tikzpicture}[y=0.80pt, x=0.80pt, yscale=-1.000000, xscale=1.000000, inner sep=0pt, outer sep=0pt]

\node at (50,95) {$\CCC^{\sbullet}$};

\node at (40,160) {\scalebox{1.1}{$\scriptstyle e_0$}};

\node at (74,150) {\scalebox{1.1}{$\scriptstyle e_1^{\sbullet}$}};
\node at (105,98) {\scalebox{1.1}{$\scriptstyle e_2^{\sbullet}$}};
\node at (71,41) {\scalebox{1.1}{$\scriptstyle e_3^{\sbullet}$}};

\node at (82,166) {\scalebox{.9}{$\scriptstyle q_0$}};
\node at (114,113) {\scalebox{.9}{$\scriptstyle p_1$}};
\node at (82,57) {\scalebox{.9}{$\scriptstyle q_1$}};

\node at (90,122) {\scalebox{.7}{\rotatebox{60}{$\scriptstyle m_1=5$}}};
\node at (100,135) {\scalebox{.7}{\rotatebox{60}{$\scriptstyle n_1=3$}}};

\node at (73,75) {\scalebox{.7}{\rotatebox{-60}{$\scriptstyle r_1=3$}}};
\node at (90,74) {\scalebox{.7}{\rotatebox{-60}{$\scriptstyle s_1=2$}}};

\node at (89,16) {\scalebox{.7}{\rotatebox{60}{$\scriptstyle m_2=1$}}};
\node at (98,25) {\scalebox{.7}{\rotatebox{60}{$\scriptstyle n_2=1$}}};

\begin{scope}[shift={(0,0)}]% layer1
  % flowRoot4196
  \path[fill=black,line join=miter,line cap=butt,line width=0.800pt]
    (0.0000,0.0000) node[above right] (flowRoot4196) {};

  \begin{scope}[shift={(-2.23262,-36.83816)}]% g6221
    % use4160-43-6
    \path[draw=black,fill=black,line join=round,line cap=rect,miter limit=4.00,line
      width=0.845pt] (74.6116,203.7468) circle (0.0745cm);

    % path4281-8-1
    \path[draw=black,line join=miter,line cap=butt,even odd rule,line width=1.739pt]
      (74.3307,203.7468) -- (9.1025,203.7468);

    % circle4589
    \path[draw=black,fill=black,line join=round,line cap=rect,miter limit=4.00,line
      width=0.845pt] (8.6844,203.7468) circle (0.0745cm);

    % path4281-8-7
    \path[draw=black,dash pattern=on 0.43pt off 0.43pt,line join=miter,line
      cap=butt,miter limit=4.00,even odd rule,line width=1.739pt] (74.4119,203.9608)
      -- (107.0260,149.6040) -- (74.4119,95.2472) -- (107.0260,40.8903);

    % use4160-43-7-3-7-34
    \path[xscale=-1.000,yscale=1.000,draw=black,fill=black,line join=round,line
      cap=rect,miter limit=4.00,line width=0.845pt] (-75.5434,95.5471) circle
      (0.0745cm);

    % use4160-43-7-40-5-0
    \path[xscale=-1.000,yscale=1.000,draw=black,fill=black,line join=round,line
      cap=rect,miter limit=4.00,line width=0.845pt] (-106.7986,150.1757) circle
      (0.0745cm);

    % use4160-43-7-0-3-78
    \path[xscale=-1.000,yscale=1.000,draw=black,fill=black,line join=round,line
      cap=rect,miter limit=4.00,line width=0.845pt] (-74.7827,203.7318) circle
      (0.0745cm);

    \begin{scope}[cm={{0.5,-0.86603,0.86603,0.5,(-76.14749,114.03488)}}]% g6000
      % rect5702-2-6-8
      \path[draw=black,fill=black,line join=round,line cap=rect,miter limit=4.00,fill
        opacity=0.294,line width=0.800pt,rounded corners=0.0000cm] (113.8481,123.0687)
        rectangle (118.8994,126.6893);

    \end{scope}
    \begin{scope}[cm={{0.5,-0.86603,0.86603,0.5,(-79.24125,106.28891)}}]% g6000-70
      % rect5702-2-6-8-1
      \path[draw=black,fill=black,line join=round,line cap=rect,miter limit=4.00,fill
        opacity=0.294,line width=0.800pt,rounded corners=0.0000cm] (113.8481,123.0687)
        rectangle (118.8994,126.6893);

    \end{scope}
    \begin{scope}[cm={{0.5,0.86603,-0.86603,0.5,(125.83798,-44.86913)}}]% g5996-0
      % rect5708-7-4
      \path[draw=black,fill=black,line join=round,line cap=rect,miter limit=4.00,fill
        opacity=0.294,line width=0.797pt] (122.5638,118.3263) -- (122.5638,115.6416)
        -- (127.8240,115.6416) -- (127.8240,118.3263) -- (127.8240,121.0110) --
        (127.8240,123.6957) -- (127.8240,126.3803) -- (122.5638,126.3803) --
        (122.5638,123.6957) -- (122.5638,121.0110) -- (122.5638,118.3263);

      % rect5702-2-9-1
      \path[draw=black,line join=round,line cap=rect,miter limit=4.00,line
        width=0.640pt,rounded corners=0.0000cm] (122.7192,119.2915) rectangle
        (127.7887,122.6836);

    \end{scope}
    \begin{scope}[cm={{0.5,-0.86603,0.86603,0.5,(-86.09611,227.47139)}}]% g5991-7
      % rect5708-5
      \path[draw=black,fill=black,line join=round,line cap=rect,miter limit=4.00,fill
        opacity=0.294,line width=0.800pt] (130.3187,113.2986) -- (130.3187,108.8649)
        -- (135.3720,108.8649) -- (135.3720,113.2986) -- (135.3720,117.7322) --
        (135.3720,122.1659) -- (135.3720,126.5995) -- (130.3187,126.5995) --
        (130.3187,122.1659) -- (130.3187,117.7322) -- (130.3187,113.2986);

      % rect5702-2-8
      \path[draw=black,line join=round,line cap=rect,miter limit=4.00,line
        width=0.640pt,rounded corners=0.0000cm] (130.3707,119.5154) rectangle
        (135.4401,122.9075);

      % rect5702-2-6-2
      \path[draw=black,line join=round,line cap=rect,miter limit=4.00,line
        width=0.640pt,rounded corners=0.0000cm] (130.2810,112.4392) rectangle
        (135.3324,116.0598);

    \end{scope}
    \begin{scope}[cm={{0.5,-0.86603,0.86603,0.5,(-73.96964,232.66995)}}]% g5996-4
      % rect5708-7-8
      \path[draw=black,fill=black,line join=round,line cap=rect,miter limit=4.00,fill
        opacity=0.294,line width=0.797pt] (122.5638,118.3263) -- (122.5638,115.6416)
        -- (127.8240,115.6416) -- (127.8240,118.3263) -- (127.8240,121.0110) --
        (127.8240,123.6957) -- (127.8240,126.3803) -- (122.5638,126.3803) --
        (122.5638,123.6957) -- (122.5638,121.0110) -- (122.5638,118.3263);

      % rect5702-2-9-2
      \path[draw=black,line join=round,line cap=rect,miter limit=4.00,line
        width=0.640pt,rounded corners=0.0000cm] (122.7192,119.2915) rectangle
        (127.7887,122.6836);

    \end{scope}
    \begin{scope}[cm={{0.5,0.86603,-0.86603,0.5,(117.38048,-35.95117)}}]% g6172
      % rect5708-7-31
      \path[draw=black,fill=black,line join=round,line cap=rect,miter limit=4.00,fill
        opacity=0.294,line width=0.800pt] (125.7561,95.3499) -- (125.7561,93.5594) --
        (131.1892,93.5594) -- (131.1892,95.3499) -- (131.1892,97.1405) --
        (131.1892,98.9310) -- (131.1892,100.7215) -- (125.7561,100.7215) --
        (125.7561,98.9310) -- (125.7561,97.1405) -- (125.7561,95.3499);

      % rect5702-2-9-11
      \path[draw=black,line join=round,line cap=rect,miter limit=4.00,line
        width=0.640pt,rounded corners=0.0000cm] (125.9980,93.5462) rectangle
        (131.0674,96.9383);

    \end{scope}
  \end{scope}
\end{scope}

\end{tikzpicture}

	}
  \caption{Example detail of $e_0\cup\CCC^{\sbullet}$}
	\label{fig:branchdetail}
\end{figure}

\textit{Empty edge:}
To make our formulas uniform, we will adopt the convention that an empty edge of $\CCC^{\sbullet}$ will have inside multiplicity of 0, and outside multiplicity of 1. Also, if there are an odd number of edges in any branch so that the last of the consecutive pairs only contains a single element, $(e_{2\alpha-1}^{\sbullet}, - )$, then we will append an empty edge to complete the pair.

\subsubsection{Euler characteristic of structure sheaf}
Now that we have the notation in place, we first derive an expression for the Euler characteristic of the structure sheaf of a curve with only one nonempty branch, $\chi(\OO_{\CCC})$ where $\CCC=e_0\cup\CCC^{\sbullet}$, and show that it is bounded below by 1. Furthermore, we will see the restrictions that equality imposes on the multiple structures $m_i,r_i,n_i,s_i$ that can appear in such a curve. 

\begin{lemma} Let $\CCC = e_0\cup\CCC^{\sbullet}$, with $\emptyset\neq\CCC^{\sbullet}$ and $|\CCC^{\sbullet}|=2\alpha$ or $2\alpha-1$ for some positive integer $\alpha$. If $|\CCC^{\sbullet}|$ is odd then we append an empty edge to $\CCC^{\sbullet}$ in the formula below.

Then the Euler characteristic $\chi(\OO_{\CCC})$ satisfies the following equality:
\begin{align}
\chi(\OO_{\CCC})&  =   \nonumber \\
&\frac{1}{2}n_1^2
+\frac{1}{2} \sum_{i=1}^{\alpha}{\left( (r_i - m_i)(  r_i - m_i +1) \right)} \nonumber \\
& +\frac{1}{2} \sum_{i=1}^{\alpha-1}{ (n_{i+1} - s_i)^2}  
+\frac{1}{2} \sum_{i=1}^{\alpha}{ (n_i + s_i -2\min(n_i,s_i))}  \nonumber \\
& +\sum_{i=2}^{\alpha}{\left(m_{i} - \min(r_{i-1},m_{i})\right)} + \frac{1}{2}s_{\alpha}^2
\nonumber \\
&\label{eq:evencase2}
\end{align}
In particular, 
$\chi(\OO_{\CCC}) \geq 1$ with equality if and only if $n_1 = s_{\alpha} =1$ and all the summation terms are zero.
\label{lemma:evencase2}
\end{lemma}
\begin{proof}
The last statement follows since $n_1,s_{\alpha}\geq 1$ and all the other summands are non-negative. Note that the second summation is always non-negative since the two factors of each summand never have opposite signs.

We will prove Eq.~(\ref{eq:evencase2}) using induction on $\alpha$, and a rearrangement of the formula in Lemma~\ref{lemma:masterformula}.

First, suppose $|\CCC^{\sbullet}|=2\alpha$. The formula in Lemma~\ref{lemma:masterformula} for $\chi(\OO_{\CCC})$ becomes:
\begin{align}
\chi(\OO_{\CCC})=& \nonumber  \\
1+\sum_{i=1}^{\alpha}&{\left(\binom{m_i+1}{2} + \binom{n_i+1}{2} -1\right)} + \sum_{i=1}^{\alpha}{\left(\binom{r_i+1}{2} + \binom{s_i+1}{2} -1 \right)  } \nonumber\\
-m_1-&\sum_{i=1}^{\alpha}{\left(m_ir_i + \min(n_i,s_i) -1\right)} - \sum_{i=1}^{\alpha-1}{\left(s_i n_{i+1}+ \min(r_i,m_{i+1}) -1 \right)} \nonumber \\
&\label{eq:evencase}
\end{align}
The summation terms are contributions from the $e_{2i-1}$ and $e_{2i}$ edges and vertex corrections from the ${p_i}$ and ${q_i}$, respectively.

In the odd case of $|\CCC^{\sbullet}|=2\alpha-1$, the formula in Lemma~\ref{lemma:masterformula} becomes:
\begin{align}
\chi(\OO_{\CCC})&= \nonumber  \\
1+\sum_{i=1}^{\alpha}&{\left(\binom{m_i+1}{2} + \binom{n_i+1}{2} -1\right)} + \sum_{i=1}^{\alpha-1}{\left(\binom{r_i+1}{2} + \binom{s_i+1}{2} -1 \right)  } \nonumber\\
- m_1-&\sum_{i=1}^{\alpha-1}{\left(m_ir_i + \min(n_i,s_i) -1\right)} - \sum_{i=1}^{\alpha-1}{\left(s_i n_{i+1}+ \min(r_i,m_{i+1}) -1 \right)} \nonumber \\
&\label{eq:oddcase}
\end{align}

If we append an empty edge to $\CCC$ in this odd case, our convention dictates that we define $r_{\alpha}=0$ and $s_{\alpha}=1$. Then, we can rewrite Eq.~(\ref{eq:oddcase}) as:
\begin{align}
\chi(\OO_{\CCC_{2\alpha}})=& \nonumber  \\
1+\sum_{i=1}^{\alpha}&{\left(\binom{m_i+1}{2} + \binom{n_i+1}{2} -1\right)} + \sum_{i=1}^{\alpha}{\left(\binom{r_i+1}{2} + \binom{s_i+1}{2} -1 \right)  } \nonumber\\
-m_1-&\sum_{i=1}^{\alpha}{\left(m_ir_i + \min(n_i,s_i) -1\right)} - \sum_{i=1}^{\alpha-1}{\left(s_i n_{i+1}+ \min(r_i,m_{i+1}) -1 \right)} \nonumber \\
\end{align}

This is now exactly the same as the even case, Eq.~(\ref{eq:evencase}). So from here, we will assume $|\CCC^{\sbullet}|$ is even, with empty edge appended, if needed, and in either case satisfies eq.~(\ref{eq:evencase}).

To begin the induction, when $\alpha = 1$, Eq.~(\ref{eq:evencase}) reduces to
\begin{align*}
\chi(\OO_{\CCC_2}) &= 1+ \binom{m_1+1}{2} + \binom{n_1+1}{2} -1 \\
 &+ \binom{r_1+1}{2} + \binom{s_1+1}{2} -1  \\
&- m_1-\left(m_1r_1 + \min(n_1,s_1) -1\right)\\
\\
&= \frac{1}{2}n_1^2 + \frac{1}{2}(r_1-m_1)(r_1-m_1+1) \\  
&+ \frac{1}{2}(n_1+s_1-2\min(n_1,s_1))
+ \frac{1}{2}s_{1}^2,
\end{align*}
which satisfies Eq.~(\ref{eq:evencase2}).

Now suppose Eq.~(\ref{eq:evencase2}) is true for all $\CCC= e_0\cup\CCC^{\sbullet}$ with $|\CCC^{\sbullet}|=2k$ and $1\leq k\leq\alpha$. Then for any $\CCC= e_0\cup\CCC^{\sbullet}$ with $|\CCC^{\sbullet}| = 2\alpha+2$, we can write this as a union $\CCC = \CCC_{2\alpha}\cup\CCC^{\sbullet}_2$ where $\CCC_{2\alpha}=e_0\cup\CCC^{\sbullet}_{2\alpha}$ contains the first $2\alpha$ edges of $\CCC^{\sbullet}$ and $\CCC^{\sbullet}_2$ are the remaining edges of $\CCC^{\sbullet}$. Then we have from eq.~(\ref{eq:evencase}),
\begin{align*}
\chi(\OO_{\CCC})&=  \chi(\OO_{\CCC_{2\alpha}}) 
+  \binom{m_{\alpha+1}+1}{2} +
   \binom{n_{\alpha+1}+1}{2} -1 \\
&+ \binom{r_{\alpha+1}+1}{2} + \binom{s_{\alpha+1}+1}{2} -1 \\
&-\left(s_{\alpha}n_{\alpha+1} + \min(r_{\alpha},m_{\alpha+1}) -1\right)\\
&-\left(m_{\alpha+1}r_{\alpha+1} + \min(n_{\alpha+1},s_{\alpha+1}) -1\right)\\
\\
&= \chi(\OO_{\CCC_{2\alpha}}) - \frac{1}{2}s_{\alpha}^2 
+ \frac{1}{2}(r_{\alpha+1} -m_{\alpha+1})(r_{\alpha+1} -m_{\alpha+1}+1) \\
&+ \frac{1}{2}(n_{\alpha+1} - s_{\alpha})^2
+ \frac{1}{2}(n_{\alpha+1} + s_{\alpha+1} -2\min(n_{\alpha+1},s_{\alpha+1}))\\
&+ \frac{1}{2}s_{\alpha+1}^2 + m_{\alpha+1} - \min(r_{\alpha},m_{\alpha+1})
\end{align*}
Using the inductive step, we get Eq.~(\ref{eq:evencase2})

Hence, the lemma follows.
\end{proof}

We can now formulate and prove a refined version of Lemma~\ref{lemma:structure-chi}.
\begin{lemma}
\label{lemma:anycurve}
Let $\FF\in \MU$ with $\CCC=\Supp\FF$. Then $\chi(\OO_{\CCC})\geq 1$.
Equality holds if and only if $\chi(\OO_{e_0\cup\CCC^{\sbullet}}) =\chi(\OO_{e_0\cup\CCC_{\sbullet}}) =\chi(\OO_{e_0\cup{\vphantom{\CCC}}^{\sbullet}\CCC}) =\chi(\OO_{e_0\cup{\vphantom{\CCC}}_{\sbullet}\CCC})=1$.
\end{lemma}

\begin{proof}
By symmetry, the Lemma~\ref{lemma:evencase2} calculations done for the case of $e_0\cup\CCC^{\sbullet}$, also hold for structure sheaves of subcurves $e_0\cup\CCC_{\sbullet}$, $e_0\cup{\vphantom{\CCC}}^{\sbullet}\CCC$, or $e_0 \cup{\vphantom{\CCC}}_{\sbullet}\CCC$, after an appropriate change of label.

Since $\CCC\cong e_0 \cup \CCC^{\sbullet}\cup \CCC_{\sbullet}\cup{\vphantom{\CCC}}^{\sbullet}\CCC \cup{\vphantom{\CCC}}_{\sbullet}\CCC$, in order to calculate $\chi(\OO_{\CCC})$ using Lemma~\ref{lemma:evencase2}, we only need to see what correction terms are needed at points $p_0$ and $q_0$.

First, consider $\chi(\OO|_{e_0\cup\CCC_{\sbullet}^{\sbullet}})$. 

Then from Lemma~\ref{lemma:masterformula}, the only difference in the Euler characteristic calculation comes from the difference in the contribution at $q_0$. We have:
\begin{equation}
\label{eq:generalcurvesupp}
\chi(\OO|_{e_0\cup\CCC_{\sbullet}^{\sbullet}}) = \chi(\OO|_{e_0\cup\CCC^{\sbullet}}) +\chi(\OO|_{e_0\cup\CCC_{\sbullet}}) - ns,
\end{equation}
where $n$ is the outside multiplicity of $e_1^{\sbullet}$ and $s$ is the outside multiplicity of ${e_1}_{\sbullet}$. 

From Lemma~\ref{lemma:evencase2}, we know that
\[
\chi(\OO|_{e_0\cup\CCC^{\sbullet}}) - \frac{1}{2}n^2 > 0, \quad \chi(\OO|_{e_0\cup\CCC_{\sbullet}}) -\frac{1}{2}s^2  > 0.
\]
Combining this with
\[
\frac{1}{2}(n-s)^2\geq 0 \Rightarrow \frac{1}{2}n^2 + \frac{1}{2}s^2 \geq ns,
\]
Eq.~(\ref{eq:generalcurvesupp}) becomes
\begin{equation}
\chi(\OO|_{e_0\cup\CCC_{\sbullet}^{\sbullet}}) > 0 \Rightarrow \chi(\OO|_{e_0\cup\CCC_{\sbullet}^{\sbullet}}) \geq 1.
\label{eq:case1eqiff}
\end{equation}
Notice that Lemma~\ref{lemma:evencase2} in fact gives us that $\chi(\OO|_{e_0\cup\CCC^{\sbullet}})- \frac{1}{2}n^2\geq \frac{1}{2}$, with equality if and only if all the conditions which imply $\chi(\OO|_{e_0\cup\CCC^{\sbullet}})=1$, including $n=1$, are satisfied. Similarly, $\chi(\OO|_{e_0\cup\CCC_{\sbullet}})- \frac{1}{2}s^2\geq \frac{1}{2}$, with equality if and only if all the conditions which imply $\chi(\OO|_{e_0\cup\CCC_{\sbullet}})=1$ hold, including $s=1$. So we see that equality in the right hand side of Eq.~\ref{eq:case1eqiff} holds if and only if $\chi(\OO|_{e_0\cup\CCC^{\sbullet}}) = \chi(\OO|_{e_0\cup\CCC_{\sbullet}}) = 1$.

Recall that from Lemma~\ref{lemma:evencase2}, we know that $\chi(\OO|_{e_0\cup\CCC^{\sbullet}})\geq 1$. By the remark at the beginning of this proof, by changing labels, a similar result holds for any of the branches, so  $\chi(\OO|_{e_0\cup\CCC_{\sbullet}}) \geq 1$ also. 

Using this, we see that equality in Eq.~\ref{eq:case1eqiff} holds if and only if $\chi(\OO|_{e_0\cup\CCC^{\sbullet}}) = \chi(\OO|_{e_0\cup\CCC_{\sbullet}}) = 1$.

A similar argument holds for $\chi(\OO|_{e_0\cup {\vphantom{\CCC}}_{\sbullet}^{\sbullet}\CCC})$, and we also have $\chi(\OO|_{e_0\cup {\vphantom{\CCC}}_{\sbullet}^{\sbullet}\CCC}) \geq 1$, with equality if and only if $\chi(\OO|_{e_0\cup{\vphantom{\CCC}}^{\sbullet}\CCC}) = \chi(\OO|_{e_0\cup{\vphantom{\CCC}}_{\sbullet}\CCC}) = 1$.

Now, in general, we have $\CCC=e_0\cup\CCC_{\sbullet}^{\sbullet} \cup {\vphantom{\CCC}}_{\sbullet}^{\sbullet}\CCC$, so applying Euler characteristic on the normalization exact sequence,
\[
\chi(\OO_{\CCC})=\chi(\OO|_{e_0\cup\CCC_{\sbullet}^{\sbullet}}) + \chi(\OO|_{e_0\cup{\vphantom{\CCC}}_{\sbullet}^{\sbullet}\CCC}) - \chi(\OO_{e_0}) = \chi(\OO|_{e_0\cup\CCC_{\sbullet}^{\sbullet}}) + \chi(\OO|_{e_0\cup{\vphantom{\CCC}}_{\sbullet}^{\sbullet}\CCC}) - 1\geq 1.\]

Equality holds if and only if $\chi(\OO|_{e_0\cup\CCC_{\sbullet}^{\sbullet}}) =\chi(\OO|_{e_0\cup{\vphantom{\CCC}}_{\sbullet}^{\sbullet}\CCC}) = 1$, which proves the lemma.
\end{proof}

\begin{remark}
Lemma~\ref{lemma:structure-chi} now immediately follows from Lemma~\ref{lemma:anycurve}.
\end{remark}

\subsection{Proof of Corollary~\ref{cor:subscheme-chi}.}\label{subsection:corproof}
\begin{proof}Given $\DDD$ a closed subscheme of a curve $\CCC$ in $\Gamma$, we first claim that we may assume that $\OO_{\DDD}$ is pure 1-dimension. If not, by primary decomposition, there is a maximal pure 1-dimensional subscheme $\DDD_1 \subset \DDD$. Then we can write
\[
0\to K_0 \to \OO_{\DDD} \to \OO_{\DDD_1} \to 0,
\]
where $K_0$ is some zero dimensional sheaf. Then $\chi(\OO_{\DDD})=\chi(\OO_{\DDD_1})+\chi(K_0)\geq \chi(\OO_{\DDD_1})$ because any zero dimensional sheaf has nonnegative Euler characteristic.

We may also assume that $\DDD$ is connected. Indeed, if $\DDD=\amalg_i \DDD_i$, then $\chi(\OO_{\DDD})=\sum_i\chi(\OO_{\DDD_i})$. 

The only case left to consider that is not already covered by Lemma~\ref{lemma:anycurve} is when $\DDD$ is a connected pure 1 dimension curve in $\Gamma$ that does not contain $e_0$. But then we can define a new curve $\DDD'\coloneqq\DDD \cup e_0$ by attaching $e_0$ to a torus fixed point of valence 1 and apply Lemma~\ref{lemma:anycurve} to $\DDD'$. This gives us $1\leq\chi(\OO_{\DDD'})=\chi(\OO_{\DDD}) + 1-m_1$, where $m_1\geq1$ is the inside thickening of the chosen attaching edge. So $\chi(\OO_{\DDD})\geq 1$ in all cases as claimed.
\end{proof}

%%%%%%%%%%%%%%%%%%%%%%%%%%%%%%%
%\input{./section5}
\section{Combinatorics} \label{sec:combinatorics}
\subsection{Discussion}
We summarise the results of the previous section and show how this leads to a generating function for the naive count of curves in $\MU$. In Proposition~\ref{prop:anyrankdeg}, we showed that the sheaves in $\MU$ are torus fixed structure sheaves of curves $\CCC$ with $\chi(\OO_{\CCC})=1$. In the proof of Lemma~\ref{lemma:evencase2} and Lemma~\ref{lemma:anycurve}, we computed the constraints this imposes on the multiple structure of $\CCC$ in order for equality to hold. This leads to the following:
\begin{proposition}
\label{prop:constraint}
Let $\FF\in\MU$ and $\Supp(\FF)=\CCC =e_0\cup\CCC^{\sbullet}\cup\CCC_{\sbullet}\cup{\vphantom{\CCC}}^{\sbullet}\CCC\cup{\vphantom{\CCC}}_{\sbullet}\CCC$. Let $\{e_i\}\not\ni e_0$ be the edges of any one of the four branches of $\CCC$. Then the multiple structures of the $\{e_i\}$ satisfy the following properties. 
\begin{enumerate}
	\item The inside multiplicity of any edge that intersects $e_0$ is unrestricted.
	\item All nonzero outside multiplicities must be 1.	
	\item For each consecutive pair $(e_{2k-1},e_{2k})$, the inside multiplicities of the second edge is equal to or one less than that of the first.
	\item The inside multiplicities are non-increasing on each branch.
\end{enumerate}
\end{proposition}
\begin{proof}
By Proposition~\ref{prop:anyrankdeg}, we must have $\chi(\OO_{\CCC}) = 1$. By Lemma~\ref{lemma:anycurve}, this holds if and only if $\chi(\OO|_{C}) = 1$ for all of the subcurves $C\in\{e_0\cup\CCC^{\sbullet},\,e_0\cup\CCC_{\sbullet},\,e_0\cup{\vphantom{\CCC}}^{\sbullet}\CCC,\,e_0\cup{\vphantom{\CCC}}_{\sbullet}\CCC\}$. By symmetry, it suffices to study the constraints this imposes on any one of these branches. 

We will choose to let $\CCC = e_0\cup\CCC^{\sbullet}$ and continue to use the same notation as in Lemma~\ref{lemma:evencase2}. A consecutive pair $(e_{2k-1},e_{2k})$ has inside multiplicity $m_{k}, r_{k}$, in that order, and outside multiplicity $n_{k}, s_{k}$. We will interpret the conclusion of the lemmas to see how they imply the conditions above.

In both lemmas, $n_1$ and $s_\alpha$ must be 1 in order that $\chi(\OO|_{C}) = 1.$ 

Consider the four summation terms in Eq.~\ref{eq:evencase2}. In order for the second and third summation to be 0, we must have all $s_i=n_i$ for $1\leq i \leq\alpha$ and all $n_{i+1}=s_i$ for $1\leq i \leq\alpha-1$. Together with $n_1=1$, this implies that all $n_i=1$ and $s_i=1$ for $1\leq i \leq\alpha$.  

This shows condition (2). 

In order for the first summation to be 0, we must have either $r_i = m_i$ or $r_i + 1 = m_i$ for all $i$. This is equivalent to condition (3). 

The fourth summation term is equal to zero only when $r_{i-1}\leq m_i$ for all $i$. This, along with condition (3), gives condition (4).
\end{proof}
We would like to count the curves that satisfy these constraints. The constraints on each branch curve are independent of the other branches, so it suffices to count the possible subcurves for any one of the types $\{e_0\cup\CCC^{\sbullet},\,e_0\cup\CCC_{\sbullet},\,e_0\cup {\vphantom{\CCC}}^{\sbullet}\CCC,\,e_0\cup{\vphantom{\CCC}}_{\sbullet}\CCC\}$, and then change labels as necessary to get the counts on the other types. 

First, we count the allowed curves on some fixed branch. Since the outside multiplicities must always be 1, the only choice is in the inside multiplicities. We can represent these lengths as boxes, where the number of boxes in each row corresponds to the multiplicity of the corresponding edge, Figure~\ref{fig:young}.

\begin{figure}[th!]
  \centering
	\scalebox{1.2}{
	%made from young in inkscape saved to png,
%then cropped, then put into this file and manually added labels from pileofboxes4a

\begin{tikzpicture}[y=0.80pt, x=0.80pt, yscale=-1.000000, xscale=1.000000, inner sep=0pt, outer sep=0pt]

\begin{scope}[shift={(0,0)}]
  \node at (0,0) {\includegraphics[scale=.3]{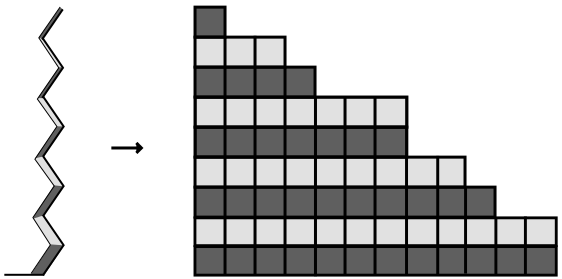}};
	
	\node at (-112,-10) {$\scriptstyle \UFs$};
	
	\node at (-27,60) {$\scriptstyle 9$};
	\node at (-15,60) {$\scriptstyle 8$};
	\node at (-3,60) {$\scriptstyle 8$};
	\node at (8,60) {$\scriptstyle 7$};
	\node at (19,60) {$\scriptstyle 6$};
	\node at (31,60) {$\scriptstyle 6$};
	\node at (42,60) {$\scriptstyle 6$};
	\node at (53,60) {$\scriptstyle 4$};
	\node at (64,60) {$\scriptstyle 4$};
	\node at (75,60) {$\scriptstyle 3$};
	\node at (86,60) {$\scriptstyle 2$};
	\node at (97,60) {$\scriptstyle 2$};
	\node at (40,75) {\scriptsize Column heights};
	
	\node at (122,-45) {$\scriptstyle 1$};
	\node at (122,-34) {$\scriptstyle 3$};
	\node at (122,-23) {$\scriptstyle 4$};
	\node at (122,-12) {$\scriptstyle 7$};
	\node at (122,0) {$\scriptstyle 7$};
	\node at (122,12) {$\scriptstyle 9$};
	\node at (120,23) {$\scriptstyle 10$};
	\node at (120,35) {$\scriptstyle 12$};
	\node at (120,47) {$\scriptstyle 12$};
	\node at (135,0) {\rotatebox{-90}{\scriptsize Thickening}};
			
		\end{scope}
\end{tikzpicture}
	}
  \caption{Multiple structure represented as a partition}
	\label{fig:young}
\end{figure}

Proposition~\ref{prop:constraint} constrains the shape of this partition. Condition (1) says that the bottom row can be any length. Condition (4) means that the rows are non-increasing in length, so we have a Young diagram. Then if we view the Young diagram as a partition via its columns rather than its rows, condition (3) forces this partition to have odd parts distinct. We can visualize this by alternating row colors to highlight consecutive pairs as in Figure~\ref{fig:young}. Here, the dark capped columns give odd parts, and they occur singly since consecutive pairs have lengths that differ by at most one.

We need to keep track of the curve class that each partition represents. Edges along a given branch of $\UFs$ alternate between pushing forward to a multiple of $[C_1]$ and to $[C_2]$. In terms of our Young diagram, this means boxes of the same color correspond to the same curve class. The specific assignment of box color to curve class depends on the branch. The difference between the number of dark and light boxes is exactly the number of odd parts that appears in the partition. 

We encode the previous discussion into a generating function. First, the number of integer partitions with \textit{only} distinct odd parts $(\ODOP)$ can be written using $q$ to track partitions and $t$ to track the number of odd parts \cite[\textsection 2.5.21]{goulden-jackson04}:
\begin{equation}
\sum\limits_{\lambda\in\ODOP} q^{|\lambda|} t^{OP(\lambda)} =\prod\limits_{n=1}^{\infty}{(1+tq^{2n-1})}.
\label{eq:odop}
\end{equation}
In this equation, $|\lambda|$ is the size of the $\ODOP$ partition $\lambda$, and $OP(\lambda)$ again denotes the number of odd parts in $\lambda$.

We are interested in partitions whose odd parts are distinct, but may have arbitrary even parts. The generating function for these odd parts distinct $(\OPD)$ partitions is thus the following modification of Eq.~(\ref{eq:odop}):
\begin{equation}
\sum\limits_{\lambda\in\OPD} q^{|\lambda|} t^{OP(\lambda)} =\prod\limits_{n=1}^{\infty}{\frac{(1+tq^{2n-1})}{(1-q^{2n})}}.
\label{eq:dop}
\end{equation}
Here, $\OPD$ are integer partitions with odd parts distinct, $|\lambda|$ is the size of the $\OPD$ partition $\lambda$, and $OP(\lambda)$ denotes the number of odd parts in $\lambda$. 

On the other hand, we can express an $(\OPD)$ partition using variables $x$ and $y$ that track the number of dark and light boxes, respectively, in our Young diagram:
\begin{equation}
\sum\limits_{\lambda\in\OPD} q^{|\lambda|} t^{OP(\lambda)} = \sum\limits_{\lambda\in\OPD} x^{\frac{1}{2}(|\lambda| +  OP(\lambda))} y^{\frac{1}{2}(|\lambda| -  OP(\lambda))}.
\label{eq:dop2}
\end{equation}
These expressions are related through the change of variables 
\[q=\sqrt{xy}, \quad t=\sqrt{\frac{x}{y}}\]
so we can rewrite the right hand side of Eq.~(\ref{eq:dop}) as:
\begin{equation}
\prod\limits_{n=1}^{\infty}{\frac{1+x^{n}y^{n-1}}{1-x^n y^n}}.
\end{equation}
So far we have restricted the discussion to one branch. For the other branches, the counts have a similar expression, but the roles of $x$ and $y$ may be reversed, depending on whether the first edge covers $[C_1]$ or $[C_2]$. Therefore, the total count of curves satisfying Proposition \ref{prop:constraint} is 
\begin{equation}
\prod\limits_{n=1}^{\infty}{\frac{(1+x^{n}y^{n-1})^2(1+x^{n-1}y^{n})^2}{(1-x^n y^n)^4}}.
\end{equation}
We have now proved the following:
\begin{proposition}\label{prop:dop3}
The number of curves $\CCC$ satisfying the constraints in Proposition~\ref{prop:constraint} can be expressed in terms of the number of partitions with distinct odd parts, namely,
\begin{equation}
\sum\limits_{d_1, d_2} \naive x^{d_1}y^{d_2}= 12\prod\limits_{n=1}^{\infty}{\frac{(1+x^{n}y^{n-1})^2(1+x^{n-1}y^{n})^2}{(1-x^n y^n)^4}}.
\label{eq:dop3}
\end{equation}
\end{proposition}
\begin{remark} The main result of the next section is to show that incorporating the Behrend function weighting into the Euler characteristic computation amounts to the following sign change:
\begin{equation}
n^0_{\beta_{d_1, d_2}}(\Xb) = (-1)^{d_1+d_2}\,\naive.
\label{eq:sign}
\end{equation}
Together with the result of Proposition~\ref{prop:dop3}, this gives Eq.~(\ref{eq:main}). This will then conclude the proof of our main result, Theorem~\ref{th:main}.
\label{rem:naive}
\end{remark}

%%%%%%%%%%%%%%%%%%%%%%%%%%%%%%%
%\input{./section6}
\section{Computing the Behrend function weighted Euler Characteristic} \label{sec:behrend function}
In this section we prove in Proposition~\ref{prop:behrend} that the naive and Behrend function weighted Euler characteristics are related by a sign change as discussed in Remark~\ref{rem:naive}. 

Recall from Proposition~\ref{prop:tfixcount} and the discussion preceding it, the sheaves in $M=M^{\Xb}_{\beta}$ are scheme-theoretically supported on fibers of $\Xb\to\PP^1$, so, in particular, we have a fibration $\rho:M\to\PP^1$. The group scheme $X^0\to\PP^1$ acts on $M\to\PP^1$, preserving the Calabi-Yau form and the symmetric obstruction theory, and hence the Behrend function. So we only need to consider isolated fixed points in $[\FF]\in M$, where $\FF$ is supported on $\Fsing$. 

In order to compute the Behrend function $\nu_{M}$, however, we need to study infinitesimal deformations into the whole space $\Xb$. The tangent space to a fixed sheaf $[\FF]\in M$, is given by
\[
T_{[\FF]}M = (\Ext^1_{\Xb})_0(\FF,\FF).
\]
Since we are only considering sheaves supported on $\Fsing$, it suffices to compute this on $\Fhat$, the formal completion of $\Xb$ along $\Fsing$, which is a formal toric Calabi-Yau threefold. The $\CC^*$-action of $\Fhat$ then allows us to use the Behrend-Fantechi result \cite[Corollary 3.5]{behrend-fantechi08}.

More precisely, we have the following definition.
\begin{definition}
Let $\Mhat$ be the formal scheme:
\[\Mhat \text { is the formal completion of } M \text { along } \Msing.\]
\end{definition}
Then the Behrend function satisfies \cite{jiang17},
\[
(\nu_{M})|_{\Msing} =(\nu_{\Mhat})|_{\Msing}.
\]
Recall that the action of the torus $T\cong\CC^*\times \CC^*$ on $\Fsing$ came from the group scheme action on $\Xb$. This torus action can be extended to an action on $\Fhat$ \cite[Lemma 4.5]{bryan19}. As a consequence, $\Mhat\subset M$ inherits a $T$ action since $\Mhat$ only depends on $\Fhat$. Furthermore, this action is shown to preserve the symmetric obstruction theory on $\Mhat$.

In the following Lemma~\ref{lemma:sot_preserved}, we show that the symmetric obstruction theory on $\Mhat$ is also equivariant with respect to the group action induced from $P$. Then using \cite[Corollary 3.5]{behrend-fantechi08}, the Behrend function weighted Euler characteristic of the moduli space depends only on the parity of the dimension of the tangent space at the fixed points $\MhatTP$ of both actions,
\[
e(\Mhat,\nu_{\Mhat}) = \sum_{[\FF]\in \MhatTP}{(-1)^{\dim \Ext^1_{\Fhat}(\FF,\FF)}}.
\]
So all that will be left to do is to determine $\dim \Ext^1(\FF,\FF)$ (mod 2) at the fixed points $[\FF]\in \MhatTP$.
\begin{lemma}
The action of $P$ extends to an action on $\Mhat$. Furthermore, the symmetric obstruction theory on $\Mhat$ is equivariant with respect to this action.
\label{lemma:sot_preserved}
\end{lemma}

\begin{proof}
The action of $P\cong\CC^*\times \CC^*$ on $\Msing$ came from tensoring by degree 0 line bundles $L_{\bm{\mu}}$ supported on $\Fsing$. By the same arguments as in Section~\ref{sec: equiv}, we also have $P\cong\CC^*\times \CC^*\subset \Pic^0(\Fhat)$. This induces an action of $P$ on the moduli space $\Mhat$ as follows.

Given some $\bm{\mu}\in P$ corresponding to the flat line bundle $L_{\bm{\mu}}$ on $\Fhat$, let $\LL_{\bm{\mu}} \coloneqq p_2^*L_{\bm{\mu}}$ where $p_i$ is projection to the $i$-th factor. Let $\EE$ be the universal sheaf over $\Mhat \times \Fhat$. 
\[ 
\begin{tikzcd}[row sep=small, column sep=small]
&\EE \arrow[d]\\
 & {\Mhat \times \Fhat}  \arrow[dl,"p_1"]\arrow[dr,swap,"p_2"]  \\ 
\Mhat&& \Fhat
\end{tikzcd}
\]
If we tensor $\EE$ by $\LL_{\bm{\mu}}$, this induces a map $\phi_{\mu}:\Mhat\to \Mhat$ by the universal property of $\EE$ as in the diagram below. This gives an action of $P$ on $\Mhat$ with $\EE$ as an $P$-equivariant sheaf.
\[ 
\begin{tikzcd}
\phi_{\mu}^*\EE\cong\EE\otimes \LL_{\bm{\mu}} \arrow[r] \arrow[d]  & \EE \arrow[d] \\ 
\phantom{\phi_{\mu}^*\EE}\Mhat \times \Fhat \arrow[r, "\phi_{\mu}"]& \Mhat \times \Fhat.
\end{tikzcd}
\]
From
\[
\HHom(\EE\otimes \LL_{\bm{\mu}},\EE\otimes \LL_{\bm{\mu}}) \cong \HHom(\EE,\EE\otimes \LL_{\bm{\mu}}\otimes \LL_{\bm{\mu}}^{\vee}) \cong  \HHom(\EE,\EE),
\]
we get the canonical isomorphism
\[
R\HHom(\EE\otimes \LL_{\bm{\mu}},\EE\otimes \LL_{\bm{\mu}}) \cong R\HHom(\EE,\EE),
\]
and thus
\[
R\HHom(\phi_{\mu}^*\EE,\phi_{\mu}^*\EE) \cong R\HHom(\EE,\EE).
\]

This implies that the shifted cone $\FF$ of the trace map $R\HHom(\EE,\EE) \to \OO_{\Mhat \times \Fhat}$ in $D(\OO_{\Mhat \times \Fhat})$ is preserved by $P$. 
\[ 
\begin{tikzcd}
 & \OO_{\Mhat \times \Fhat}  \arrow[dl, swap, "+1"]  \\ 
\FF \arrow[rr]&& R\HHom(\EE,\EE)\arrow[ul,swap,"tr"]
\end{tikzcd}
\]
All the constructions of the obstruction theory \cite[Lemma 2.2]{behrend-fantechi08}
\[E\coloneqq R(p_1)_{*}R\HHom(\FF,\omega_{\Fhat})[2] \to L_{\Mhat},\]
as well as the nondegenerate symmetric bilinear form $\theta:E\to E^{\vee}[1]$ which is induced from $\omega_{\Fhat}\cong \OO_{\Fhat}\to \OO_{\Fhat}$, are also equivariant. Hence the $P$-action is equivariant and symmetric, and preserves the symmetric obstruction theory on $\Mhat$.
\end{proof}

\subsection{Relating deformations of sheaves on $\Fhat$ and $\UFhat$}
We will show that the dimension of $\Ext^1(\FF, \FF)$ for the fixed points $[\FF]\in \MsingTP$ has the same parity whether considered as sheaves on $\Fhat$ or on $\UFhat$. This implies that their Behrend function contributions to the Euler characteristic are the same, so we may calculate this on $\UFhat$.  We regard the fixed points as sheaves on the formal schemes, pushed forward under the respective inclusions $\Fsing\hookrightarrow \Fhat$ and $\UFs \hookrightarrow \UFhat$.

In Proposition~\ref{prop:anyrankdeg}, we showed that sheaves in $\MU$ were possibly non-reduced structure sheaves $\OO_{\CCC}$ of certain types of curves in $\UFhat$. As explained in section \ref{sec: equiv}, 
this corresponds to a point in $\MsingTP$ by 
\[
\MsingTP\ni\FF \longleftrightarrow \FFF_0 = \OO_{\CCC}\in \MU,
\]
where the correspondence is given as
\begin{equation}
\pr_*\FFF_0 = \FF \text{ and }
\pr^*\FF = \FFF=\bigoplus_{k,l\in\ZZ^2}(e_1^k e_2^l )^*\FFF_0
\label{eq:atomictranslates}
\end{equation}
with the $G\coloneqq\ZZ\times\ZZ$ action on $\Coh(\UFhat)$ covering the deck transformations.

\begin{proposition}
For any $\FF\in \MsingTP$, let $\OO_{\CCC} \in  \MU$ be the corresponding stable sheaf on $\UFhat$ so that $\pr_*(\OO_{\CCC}) = \FF$. Then 
\[\Ext^1(\FF,\FF) \cong \Ext^1(\OO_{\CCC},\OO_{\CCC}) \oplus \CC^2.\]
In particular, the dimensions of the deformation spaces have the same parity.
\label{prop:extdowntoup}
\end{proposition}
\begin{proof}
Fix $\FF\in \MsingTP$. Recall the proof of Proposition~\ref{prop:atomic}. Under the general categorical equivalence of sheaves on $\Fsing$ with $G\coloneqq\ZZ\times\ZZ$ equivariant sheaves on $\UFhat$, deformations of $\FF$ correspond to deformations of the corresponding $G$-sheaf $\pr^*\FF = (\FFF, \phi_1, \phi_2)$, $[\phi_1,\phi_2]=0$. We can separate the deformations of the sheaf from the deformations of the lift of the action by considering the linear map between deformation spaces which forgets the equivariant part of the sheaf:
\[ 0\to \Ker \to \Def (\FFF, \phi_1, \phi_2) \to \Def (\FFF_0)\to 0.\]
The kernel consists of deformations of the linear maps $\phi_i\in \Hom(\FFF, e_i^*\FFF)$. These are given by pairs, 
\[ \{(\phi_1+\epsilon\eta_1 ,\phi_2+\epsilon\eta_2 )\}, \qquad \eta_i\in \Hom(\FFF, e_i^*\FFF), \ \epsilon^2=0,
\] 
which cover the group action, so $[\phi_1 +\epsilon\eta_1,\phi_2+\epsilon\eta_2]=0$.
In other words
\[\Ker =\Def(\phi_1,\phi_2) = \{(\eta_1,\eta_2 )| [\eta_1,\phi_2] + [\eta_2,\phi_1]=0\}
\]

From Proposition~\ref{prop:atomic} the sheaves in $\MU$ are of the special form satisfying Eq. (\ref{eq:atomictranslates}), and so $\FFF \cong e_i^*\FFF$. Observe that in $\Coh^{G}\UFhat$ we can re-index, and then by equivariance and stability, we get
\[\Hom_{G}(\FFF, e_i^*\FFF) \cong \Hom_{G}(\FFF, \FFF) \cong \Hom(\FFF_0, \FFF_0) \cong \CC.\]
So the commutator relation is trivial and $\{(\eta_1,\eta_2)\} = \CC\times\CC$.
\end{proof}

\subsection{Computing deformations on {\mdseries$\UFhat$}} 
Let ${\CCC}$ be the support of a point $[\OO_{\CCC}]\in \MU$. To apply Proposition~\ref{prop:extdowntoup}, we need to calculate the parity of the dimension of $\Ext^1(\OO_{\CCC},\OO_{\CCC})$. We will do this by reducing our computation to the result in \cite[Theorem 2]{mnop1}. We work in an ambient toric Calabi-Yau threefold, which we describe below.

For a fixed degree $\beta = (d_1,d_2,1)$, the support ${\CCC}$ of any stable sheaf in $\MU$ is contained in a finite type region of $\UFhat$. Following the discussion in Subsection~\ref{geometry}, such a region is formally locally isomorphic to some ambient smooth finite type toric Calabi-Yau threefold $W\subset \mathfrak{W}$, whose fan consists of the cones over the finitely many tiles of Figure~\ref{fig:triangle_tile} that contain $\Supp(\CCC)$. We may thus compute the infinitesimal deformations of sheaves in $\MU$ by considering them as sheaves on $W$.
\begin{definition} Let $W$ be a smooth finite type toric Calabi-Yau threefold formally locally isomorphic to a formal neighborhood of $\Supp(\CCC)$ for $\CCC\in\MU$.
\end{definition}
For the remainder of this section, we will work on the space $W$ for the computations of $\Ext$ and $\Hom$ groups. With this understanding, we will often suppress the subscript $W$ and write $\Ext\coloneqq\Ext_{W}$ and $\Hom\coloneqq\Hom_{W}$.

Furthermore, in \cite{mnop1} Maulik et al consider ideal sheaves, whereas we are interested in structure sheaves. So we will also need the following Lemma~\ref{lemma:icic}, but we defer its proof until after Proposition~\ref{prop:behrend}.
\begin{lemma}
$\Ext^1_{W}(\OO_{\CCC},\OO_{\CCC}) \cong \Ext^1_{W}(\II_{\CCC},\II_{\CCC})$, where $\II_{\CCC}$ is the ideal sheaf of ${\CCC}$ in $W$.
\label{lemma:icic}
\end{lemma}
\begin{proof}
See following the proof of Proposition~\ref{prop:behrend}.
\end{proof}
\begin{proposition}\label{prop:behrend}
The dimension of $\Ext^1(\OO_{\CCC},\OO_{\CCC})$ is $d_1+d_2 \mod 2.$
\end{proposition}
\begin{proof}  
We apply the formula of \cite[Theorem 2]{mnop1} to compute the dimension of the tangent space. This result was proved with $T^3\coloneqq (\CC^*)^3$-equivariant cohomology. By equivariant Serre duality and restriction to the Calabi-Yau torus $T \cong (\CC^*)^2 \subset T^3$ we get the equality:
\[
\frac{e(\Ext_{W}^1(\II_{\CCC},\II_{\CCC}))}{e(\Ext_{W}^2(\II_{\CCC},\II_{\CCC}))}\Bigg|_{T} = (-1)^{\dim \Ext^1_W (\II_{\CCC},\II_{\CCC}) },
\]
where
\[{\dim\Ext^1_{W}(\II_{\CCC},\II_{\CCC})} 
\equiv {\chi(\OO_{\CCC}) +\sum_{C_i}{m_{C_i}d_{C_i}}} \quad(\text{mod 2}), \quad {\CCC}=\cup {C_i}.
\]
The sum is taken over irreducible components $C_i$ in the support, each of which has normal bundle $\OO(-m_i)\oplus \OO(-2+m_i)$ and length $d_{C_i}$. In our situation, we have that $\chi({\CCC})=1$, all $m_{C_i}=1$, and the total degree is ${d_1+d_2+1}$.
\end{proof}

To complete the proof of Proposition~\ref{prop:behrend}, we need to prove Lemma~\ref{lemma:icic}. After some preliminary calculations, we will prove Lemma~\ref{lemma:icic} by showing two separate isomorphisms,
$\Ext^1(\OO_{\CCC}, \OO_{\CCC}) \cong \Hom(\II_{\CCC},\OO_{\CCC})$ and $
\Hom(\II_{\CCC},\OO_{\CCC})\cong \Ext^1(\II_{\CCC},\II_{\CCC})$, which we deduce from different long exact sequences.

We begin with some preliminary observations that follow from our geometry.
\begin{lemma} \label{lemma:homologicalgeometry}
With the notation as above, we have the following equations.
\begin{equation}\label{eq:homoxoc}
\Hom(\OO_{W} ,\OO_{\CCC})=\Hom(\OO_{\CCC} ,\OO_{\CCC})=\CC,
\end{equation}
\begin{equation}\label{eq:ext1oxoc}
\Ext^1(\OO_{W},\OO_{\CCC}) = 0,
\end{equation}
\begin{equation}\label{eq:homintoOx-2}
\Ext^2(\OO_{\CCC},\OO_{W})=\Ext^1(\OO_{\CCC},\OO_{W})=\Hom(\OO_{\CCC} ,\OO_{W}) =0,
\end{equation}

\end{lemma}
\begin{proof} 
Eq.~(\ref{eq:homoxoc}) follows from stability,
\[\Hom(\OO_{W} ,\OO_{\CCC})=\Hom(\OO_{\CCC} ,\OO_{\CCC}) = H^0(\OO_{\CCC}) = \CC.
\]
Then since our support curve $\CCC$ is assumed to have $\chi(\OO_{\CCC})=1$, we get Eq.~(\ref{eq:ext1oxoc}),
\[\Ext^1(\OO_{W},\OO_{\CCC}) = H^1(\OO_{\CCC}) = 0.\]
Next, for Eq.~(\ref{eq:homintoOx-2}) we compute 
\begin{align*}
\Ext^2(\OO_{\CCC},\OO_{W}) &=  \Ext^1(\OO_{W},\OO_{\CCC} \otimes K_{W})^{\vee} &\text{by $T^3$-equivariant Serre duality}\\
&= \Ext^1(\OO_{W},\OO_{\CCC})^{\vee}&\text{since $W$ is Calabi-Yau}\\
&=H^1(\OO_{\CCC})^{\vee}  & \\
&=0  &\text{by Eq.~(\ref{eq:ext1oxoc}).}
\end{align*}
Similarly, we use $T^3$-equivariant Serre duality and $W$ being a Calabi-Yau threefold for the other equations:
\begin{align*}
\Ext^1(\OO_{\CCC},\OO_{W}) &=H^2(\OO_{\CCC})^{\vee}=0, &\text{since $C$ has dimension 1.}\\
\Hom(\OO_{\CCC} ,\OO_{W}) &= H^3(\OO_{\CCC})^{\vee}=0 &\text{since $C$ has dimension 1.}
\end{align*}
\end{proof}

The first isomorphism we need to prove is the following.
\begin{lemma}
\label{lemma:exthomforOc}
Let the notation be as above. Then,
\begin{equation}
\Ext^1(\OO_{\CCC}, \OO_{\CCC}) \cong \Hom(\II_{\CCC},\OO_{\CCC}) 
\label{eq:exthomforOc}
\end{equation}
\end{lemma}

\begin{proof}

We start with the exact sequence on $W$:
\begin{equation}
0\to\II_{\CCC} \to\OO_{W} \to \OO_{\CCC} \to 0. 
\label{eq:standardidealseq}
\end{equation}
If we apply $\Hom(\,\bigcdot\, ,\OO_{\CCC})$ to Eq.~(\ref{eq:standardidealseq}), we get the long exact sequence:
\begin{align}
0\to &\Hom(\OO_{\CCC} ,\OO_{\CCC})  \to\Hom(\OO_{W},\OO_{\CCC}) \to \Hom(\II_{\CCC},\OO_{\CCC}) \to  \nonumber\\
\to &\Ext^1(\OO_{\CCC} ,\OO_{\CCC})  \to\Ext^1(\OO_{W},\OO_{\CCC}) \to \cdots
\label{eq:homtoOc} 
\end{align}
Using Eq.~(\ref{eq:homoxoc}) and Eq.~(\ref{eq:ext1oxoc}) from Lemma~\ref{lemma:homologicalgeometry} in the long exact sequence Eq.~(\ref{eq:homtoOc}) gives
\[ 
\Hom(\II_{\CCC},\OO_{\CCC})\cong\Ext^1(\OO_{\CCC}, \OO_{\CCC})
\] 
as required.
\end{proof}

The second isomorphism is below.
\begin{lemma}
Let the notation be as above. Then,
\[
\Hom(\II_{\CCC},\OO_{\CCC})\cong \Ext^1(\II_{\CCC},\II_{\CCC})
\]
\label{lemma:exthomforIc}
\end{lemma}

\begin{proof}
We start with the same exact sequence on $W$ as above:
\begin{equation}
0\to\II_{\CCC} \to\OO_{W} \to \OO_{\CCC} \to 0. 
\label{eq:standardidealseq2}
\end{equation}
This time we apply $\Hom(\,\bigcdot\, ,\OO_{W})$ to Eq.~(\ref{eq:standardidealseq2}) to get the long exact sequence:
\begin{align}
0\to &\Hom(\OO_{\CCC} ,\OO_{W})  \to\Hom(\OO_{W},\OO_{W}) \to \Hom(\II_{\CCC},\OO_{W}) \to  \nonumber\\
\to &\Ext^1(\OO_{\CCC} ,\OO_{W})  \to\Ext^1(\OO_{W},\OO_{W}) \to \Ext^1(\II_{\CCC},\OO_{W}) \to 
\label{eq:homintoOx} \\
\to &\Ext^2(\OO_{\CCC} ,\OO_{W}) \to \cdots\nonumber 
\end{align}

Applying Lemma~\ref{lemma:homologicalgeometry} to the long exact sequence Eq.~(\ref{eq:homintoOx}) yields two isomorphisms,
\begin{align}
\Hom(\OO_{W},\OO_{W}) &\cong \Hom(\II_{\CCC},\OO_{W}),\label{eq:homintoOx-A} \\
\Ext^1(\OO_{W},\OO_{W}) &\cong \Ext^1(\II_{\CCC},\OO_{W}). \label{eq:homintoOx-B}
\end{align}

Define the ring $R$ as
\[
R\coloneqq \Hom(\OO_{W},\OO_{W}) =H^0(\OO_{W}).
\]
Then Eq.~(\ref{eq:homintoOx-A}) gives
\begin{equation}
\Hom(\II_{\CCC},\OO_{W}) \cong R. \label{eq:homicox}
\end{equation}
The isomorphism $\Hom(\II_{\CCC},\OO_{W})\cong R$ identifies the function $f\in R$ with the homomorphism given by multiplication by $f$, 
\[\II_{\CCC}\xrightarrow{\bullet f}\OO_{W}.\]

Also let $\Waff$ be the affinization of $W$,
\[
\Waff \coloneqq \Spec R = \Spec H^0(\OO_{W}).
\]
In terms of the toric fans, the fan of $W$ is a refinement of that of $\Waff$, and
\[
W \xrightarrow{\pi} \Waff 
\] is a projective morphism.
Hence, 
\begin{align}
\Ext^1(\OO_{W},\OO_{W})  &= H^1(W,\OO_W) \nonumber \\
&= H^1(\Waff,\pi_*\OO_{W}) &\text{by vanishing of higher direct image sheaves}\nonumber \\
&= H^1(\Waff,\OO_{\Waff}) &\nonumber \\
&= 0  &\text{since $\Waff$ is affine}. \label{eq:h1oxvanish}
\end{align}

Using  Eq.~(\ref{eq:h1oxvanish}) in Eq.~(\ref{eq:homintoOx-B}), we get
\begin{equation}
\Ext^1(\II_{\CCC},\OO_{W}) = 0.
\label{eq:ext1icoxvanish}
\end{equation}

Finally, we apply $\Hom(\II_{\CCC},\,\bigcdot\, )$ to Eq.~(\ref{eq:standardidealseq2}) to get the long exact sequence:
\begin{align}
0\to &\Hom(\II_{\CCC}, \II_{\CCC}) \to\Hom(\II_{\CCC},\OO_{W}) \to \Hom(\II_{\CCC},\OO_{\CCC}) \to  \nonumber\\
\to &\Ext^1(\II_{\CCC}, \II_{\CCC}) \to\Ext^1(\II_{\CCC},\OO_{W}) \to \cdots \label{eq:homIcto} 
\end{align}
Using Eq.~(\ref{eq:homicox}), we have
\begin{equation*}
\Hom(\II_{\CCC}, \II_{\CCC})\hookrightarrow  R\cong \Hom(\II_{\CCC},\OO_{W}).
\end{equation*}
But we also have $R\subset\Hom(\II_{\CCC}, \II_{\CCC})$ since any $f\in R$ gives a homomorphism $\II_{\CCC}\xrightarrow{\bullet f}\II_{\CCC}$. So we have,
\begin{equation}
\Hom(\II_{\CCC}, \II_{\CCC})\cong R.
\label{eq:icicisR}
\end{equation}

Now using Eq.~(\ref{eq:homicox}), Eq.(\ref{eq:icicisR}), and Eq.(\ref{eq:ext1icoxvanish}) in the long exact sequence Eq.(\ref{eq:homIcto}), we conclude that
\[
\Hom(\II_{\CCC},\OO_{\CCC}) \cong \Ext^1(\II_{\CCC}, \II_{\CCC})
\]
\end{proof}

\begin{proof}[Proof (of Lemma~\ref{lemma:icic}).] Follows immediately from Lemma~\ref{lemma:exthomforOc} and Lemma~\ref{lemma:exthomforIc}.
\end{proof}

%%%%%%%%%%%%%%%%%%%%%%%%%%%%%%%

%\bibliography{mainbib}
%\bibliographystyle{plain}

\end{document}